\documentclass[10pt]{amsart}
\usepackage{amsfonts,amscd,amsthm,amsgen,amsmath,amssymb}
\usepackage[all]{xy}
\usepackage[vcentermath]{youngtab}
\usepackage{graphicx}
\usepackage{hyperref}

\newtheorem{theorem}{Theorem}[section]
\newtheorem{lemma}[theorem]{Lemma}
\newtheorem{corollary}[theorem]{Corollary}
\newtheorem{proposition}[theorem]{Proposition}

\theoremstyle{definition}
\newtheorem{definition}[theorem]{Definition}
\newtheorem{example}[theorem]{Example}

\theoremstyle{remark}
\newtheorem{remark}[theorem]{Remark}

\numberwithin{equation}{section}

\begin{document}

\title[Equivariant cohomology of slice groupoids]{Equivariant cohomology of slice groupoids}
\author{Zhenxi Huang}


\begin{abstract}
Let $G$ be a compact Lie group, let $M$ be a smooth manifold equipped with a $G$-action, and let $\pi:M\rightarrow M/G$ denote the quotient map. Then the action is encoded by the action groupoid $G\ltimes M$. Moreover, the slice theorem implies that for each point $y\in M/G$, there is a neighborhood $U_y\subset M/G$ of $y$ such that \begin{equation*} \pi^{-1}(U_y)\cong G\times_{G_x}S_x, \end{equation*} where $x$ is a point in the orbit corresponding to $y$, $S_x$ is a slice through $x$, and $G_x$ is the isotropy group of $x$. An alternative approach to describing group actions on spaces is through the language of groupoids. Local properties of Lie groupoids are often studied via linearization theorems (see, e.g., \cite{Wei,Zu,Fe,Sm}).

One can compute the equivariant cohomology $H_G^*(\pi^{-1}(U_y);\mathbb{R})$ of $\pi^{-1}(U_y)$ using either the Weil model or the Cartan model. By equivariant homotopy theory, the equivariant cohomology groups $H_G^*(\pi^{-1}(U_y);\mathbb{R})$ and $H_{G_x}^*(S_x;\mathbb{R})$ are isomorphic.

In this paper, we explicitly construct reduction chain maps between the Weil models, and respectively between the Cartan models, of $(\pi^{-1}(U_y),G)$ and $(S_x,G_x)$, and prove that they induce isomorphisms in equivariant cohomology. We then introduce the notion of slice (or locally linearizable) groupoids, which are locally modeled on Lie group actions on manifolds together with gluing data. Several examples and applications are discussed. In the last section, we generalize equivariant cohomology to these groupoids using sheaf-theoretic methods. We further show that when a slice groupoid arises from an action Lie groupoid, its equivariant cohomology coincides with the equivariant cohomology of the corresponding Lie group action on a manifold.

\end{abstract}
\thanks{Email address: huangzhenxi@jnu.edu.cn}
\thanks{	
	This work was supported by: 1. Guangzhou Municipal Science and Technology Bureau, Guangzhou Key Research and Development Program, Program No. 2024A04J3560; 2. Fundamental Research Funds for the Central Universities, Program No. 11623340}

\date{\today}



\maketitle


\section{Equivariant cohomology}\label{sec:intro}

\subsection{Introduction}
The topology and cohomology of a $G$-space $M$ and of the quotient space $M/G$ arise in many branches of geometry. Let $M$ be a manifold. In general, if $G$ does not act freely on $M$, the quotient space may be singular even though $M$ is smooth. From the viewpoint of equivariant homotopy theory, one is naturally led to the equivariant cohomology of $M$. Throughout this paper, group actions on manifolds are taken to be left actions. Let $E_G$ be a contractible space equipped with a free right $G$-action (see \cite{Gui1} for the existence of $E_G$). For a left $G$-space $M$, define the Borel construction by \begin{equation*} E_G\times_G M = (E_G\times M)/G, \end{equation*} where $G$ acts freely from the right on $E_G\times M$ by \begin{equation*} (p,x)\cdot g = (pg,g^{-1}\cdot x). \end{equation*} Equivalently, the quotient relation is \begin{equation*} [p,x] = [pg,g^{-1}\cdot x]. \end{equation*} The equivariant cohomology of $M$ is defined by \begin{equation*} H_G^*(M;\mathbb{R}) = H^*(E_G\times_G M;\mathbb{R}). \end{equation*} Throughout this paper, all cohomology groups are taken with coefficients in $\mathbb{R}$, unless otherwise stated. Accordingly, $H_G^*(M)$ denotes $H^*(E_G\times_G M;\mathbb{R})$, and all de Rham, Weil, and Cartan complexes are considered over $\mathbb{R}$.

Since the right $G$-action on $E_G$ is free, the induced right action \begin{equation*} (p,x)\cdot g=(pg,g^{-1}\cdot x) \end{equation*} on $E_G\times M$ is free. In particular, if $G$ acts freely on $M$, then $M/G$ is a manifold, and the equivariant cohomology is naturally isomorphic to the ordinary cohomology of $M/G$ \cite{Bor}. One may also consider a de Rham realization of the equivariant cohomology. Formally this is related to differential forms on the Borel construction $E_G\times_G M$. Since $E_G$ is generally infinite dimensional when $G$ is nontrivial, one usually uses finite-dimensional approximations, or an equivalent algebraic model such as the Weil or Cartan model.

A convenient way to compute the equivariant cohomology of $M$ is to replace the complex $\Omega^*(E_G\times_G M)$ by an algebraic model. The Weil model is the complex of basic elements $[\Omega(M)\otimes\Lambda(\mathfrak{g}^*)\otimes S(\mathfrak{g}^*)]_{\mathrm{bas}}$, while the Cartan model is the complex of invariant elements $[\Omega(M)\otimes S(\mathfrak{g}^*)]^G$, equipped with an appropriate differential. When $G$ is compact, the cohomology groups of both complexes are isomorphic to equivariant cohomology. The explicit constructions of the Weil and Cartan models are given as follows; see \cite{Gui1}.

Let $G$ be an $m$-dimensional Lie group, let $\{\xi_1,\ldots,\xi_m\}$ be a basis of its Lie algebra $\mathfrak{g}$, and let $\{x^1,\ldots,x^m\}$ be the dual basis of $\mathfrak{g}^*$. We use the Einstein summation convention $[\xi_i,\xi_j]=c_{ij}^k\xi_k$ for the Lie bracket. 
If $H\subset G$ is a closed Lie subgroup, choose $\{\xi_1,\ldots,\xi_t\}$ to be a basis of its Lie algebra $\mathfrak{h}$. Then
\begin{equation*}
	c_{ij}^k=0, 
\end{equation*}
when $i,j\leq t$ and $k>t$.

Let $\theta^i=x^i\otimes 1 \in \Lambda^1(\mathfrak{g}^*)\otimes S^0(\mathfrak{g}^*)$ and $z^i=1\otimes x^i\in \Lambda^0(\mathfrak{g}^*)\otimes S^1(\mathfrak{g}^*)$. Then the Weil algebra $W(\mathfrak{g}^*)=\Lambda(\mathfrak{g}^*)\otimes S(\mathfrak{g}^*)$ is generated by $\theta^1,\ldots,\theta^m,z^1,\ldots,z^m$, with $\theta^i\theta^j=(x^i\wedge x^j)\otimes 1$, $z^iz^j=1\otimes x^ix^j$, and
\begin{equation*}
\theta^{i_1}\cdots\theta^{i_p}z^{j_1}\cdots z^{j_q}
=(x^{i_1}\wedge\cdots\wedge x^{i_p})\otimes(x^{j_1}\cdots x^{j_q}).
\end{equation*} 

There are three operators on the Weil algebra. The Lie derivative on the Weil algebra is defined by
\begin{equation*}
\mathcal{L}_{{\xi_a}}\theta^b=-c^b_{ak}\theta^k
\end{equation*}
and
\begin{equation*}
\mathcal{L}_{{\xi_a}}z^b=-c^b_{ak}z^k.
\end{equation*}
The differential $\delta$ on the Weil algebra is defined by
\begin{equation*}
\delta\theta^a=z^a, \ \ \ \delta z^a=0.
\end{equation*}
The interior operator $\iota$ on the Weil algebra is defined by
\begin{equation*}
\iota_{\xi_a}\theta^b=\delta_a^b, \ \ \ \iota_{\xi_a}z^b=-c_{ak}^b\theta^k,
\end{equation*}
and these operators are extended as derivations to all of $W$. 

The degree of an element  $\theta^{i_1}\cdots\theta^{i_p}z^{j_1}\cdots z^{j_q}\in W(\mathfrak{g}^*)$ is defined to be $p+2q$. The collection of all degree $d$ elements in $W(\mathfrak{g}^*)$ is denoted by $W^d(\mathfrak{g}^*)$. From now on we use $\mathcal{L}_a$ and $\iota_a$ to replace $\mathcal{L}_{\xi_a}$ and $\iota_{\xi_a}$ to simplify the notation once the basis $\{\xi_a\}_{a=1}^m$ has been fixed.

Let $[\Omega(M)\otimes W(\mathfrak{g}^*)]_{\mathrm{bas},G}\subset\Omega(M)\otimes W(\mathfrak{g}^*)$ denote the subspace consisting of elements $\sigma\in\Omega(M)\otimes W(\mathfrak{g}^*)$ such that $\mathcal{L}_{a}\sigma=0$ and $\iota_{a}\sigma=0$. Note that $\mathcal{L}_a$ is computed by the Leibniz rule and
\begin{equation*}
\iota_{a}=\iota_a\otimes 1+ 1\otimes (-1)^{p}\iota_a
\end{equation*}
when applied to an element $\alpha^{p,q}\in\Omega^p(M)\otimes W^q(\mathfrak{g}^*)$ for $a=1,\cdots,m$. Then the Weil model is defined as the graded algebra
\begin{equation*}
	\operatorname{Weil}(M,G)=\bigoplus_{d\geq 0}\operatorname{Weil}^d(M,G)
\end{equation*}
where
\begin{equation*}
\operatorname{Weil}^d(M,G)=\bigoplus_{p+q=d}[\Omega^p(M)\otimes W^q(\mathfrak{g}^*)]_{\mathrm{bas},G}
\end{equation*}
 with the graded differential operator
\begin{equation*}
D=d\otimes 1 +(-1)^{p}1\otimes\delta.
\end{equation*}
The complex $\operatorname{Weil}(M,G)$ with differential $D$ is called the Weil model; in the rest of the paper we write $\operatorname{Weil}(M,G)$ for the differential graded algebra $(\operatorname{Weil}(M,G),D)$. We call an element $\tau\in\Omega(M)\otimes W(\mathfrak{g}^*)$ $G$-invariant if $\mathcal{L}_a\tau=0$ for every $a$, and $G$-horizontal if $\iota_a\tau=0$ for every $a$. If $\tau$ is both horizontal and invariant, then it is called a basic element. Thus the Weil model $\operatorname{Weil}(M,G)$ is the graded algebra of basic elements in $\Omega(M)\otimes W(\mathfrak{g}^*)$ equipped with the differential $D$. The operators $D$, $\iota_a$, and $\mathcal{L}_a$ satisfy the Cartan homotopy formula
\begin{equation*}
	\mathcal{L}_a=[D,\iota_a]=D\iota_a+\iota_a D.
\end{equation*} 

To construct the Cartan model, let
\begin{equation}
u^a=z^a+\frac{1}{2}c_{jk}^a\theta^j\theta^k
\end{equation}
define a change of variables. Then we have
\begin{equation*}
\begin{split}
du^a&=-c_{ij}^a\theta^iu^j\\
\mathcal{L}_{a}u^b&=-c_{ak}^bu^k\\
\iota_{a}u^b&=0.
\end{split}
\end{equation*}
Note that if $u^b$ is invariant (i.e. $\mathcal{L}_au^b=0$ for any $a$) then $u^b=z^b$.

After this change of variables, $W(\mathfrak{g}^*)$ is generated by $\{\theta^1,\ldots,\theta^m,u^1,\ldots,u^m\}$. Since $\iota_a\theta^b=\delta_a^b$, the horizontal subalgebra is $W(\mathfrak{g}^*)_{\mathrm{hor}}=1\otimes S(\mathfrak{g}^*)$, and the basic subalgebra is $W(\mathfrak{g}^*)_{\mathrm{bas}}=1\otimes S(\mathfrak{g}^*)^G\cong S(\mathfrak{g}^*)^G$, where $S(\mathfrak{g}^*)^G$ denotes the $G$-invariant subalgebra.
Let $\gamma=\sum_a\iota_a\otimes\theta^a$ be an endomorphism of $\Omega(M)\otimes W(\mathfrak{g}^*)$ such that 
\begin{equation*}
	\gamma(\omega\otimes\alpha)=\sum_{a}\iota_{a}\omega\otimes\theta^a\alpha,
\end{equation*}
Define
\begin{equation*}
\Phi=e^\gamma=1+\gamma+\frac{1}{2}\gamma^2+\frac{1}{3!}\gamma^3+\cdots,
\end{equation*}
Then $\Phi$ is an automorphism of $\Omega(M)\otimes W(\mathfrak{g}^*)$ known as the Mathai--Quillen isomorphism \cite{MQ}. In particular, $\Phi$ carries $[\Omega(M)\otimes W(\mathfrak{g}^*)]_{\mathrm{hor}}$ into $\Omega(M)\otimes [W(\mathfrak{g}^*)]_{\mathrm{hor}}$, hence we have the restriction
\begin{equation*}
\Phi:[\Omega(M)\otimes W(\mathfrak{g}^*)]_{\mathrm{bas}}\rightarrow [\Omega(M)\otimes S(\mathfrak{g}^*)]^G.
\end{equation*}
Moreover, $\Phi$ defines a differential $D_C$ on $[\Omega(M)\otimes S(\mathfrak{g}^*)]^G$ by
\begin{equation*}
D_C=\Phi D\Phi^{-1}=d\otimes 1- \sum_a\iota_a\otimes u^a.
\end{equation*}
The differential graded algebra $[\Omega(M)\otimes S(\mathfrak{g}^*)]^G$ with differential $D_C$ is called the Cartan model and is denoted by $\operatorname{Car}(M,G)$.

The Weil and Cartan models can be used to compute equivariant de Rham cohomology. This is a classical result; see \cite{Gui1}. Explicitly,
\begin{theorem}\label{thma4}
When $G$ is compact, the cohomology of the Weil model and the cohomology of the Cartan model are isomorphic to the equivariant de Rham cohomology.
\end{theorem}

\subsection{Local models} \label{sec1.2}

Now we derive local properties of the equivariant cohomology. Assume that $G$ acts properly on $M$, and let
\begin{equation*}
	\pi:M\rightarrow M/G
\end{equation*}
be the quotient map. Let $y\in M/G$, and let $U_y$ be a sufficiently small neighborhood of $y$. By the slice theorem, for such a neighborhood $U_y\subset M/G$, the inverse image $\pi^{-1}(U_y)$ is $G$-equivariantly diffeomorphic to \begin{equation*} G\times_{G_x}S_x, \end{equation*} where $x\in\pi^{-1}(y)$, $S_x$ is a $G_x$-invariant slice, and $G_x$ acts freely from the right on $G\times S_x$ by \begin{equation*} (g,z)\cdot h = (gh,h^{-1}\cdot z), \qquad h\in G_x. \end{equation*} Thus the quotient relation is \begin{equation*} [g,z] = [gh,h^{-1}\cdot z]. \end{equation*} The left $G$-action on $G\times_{G_x}S_x$ is given by \begin{equation*} k\cdot[g,z] = [kg,z]. \end{equation*} By definition, \begin{equation*} H_G^*(G\times_{G_x}S_x;\mathbb R) = H^*\bigl(E_G\times_G(G\times_{G_x}S_x);\mathbb R\bigr). \end{equation*} Since the right $G$-action on $E_G$ restricts to a free right $G_x$-action, the same space $E_G$ may also be used as a model for $E_{G_x}$. Define \begin{equation*} \begin{aligned} f: E_G\times_G(G\times_{G_x}S_x) &\longrightarrow E_G\times_{G_x}S_x,\\ [p,[g,z]] &\longmapsto [pg,z]. \end{aligned} \end{equation*} We first check that $f$ is well defined with respect to the $G$-quotient. For $k\in G$, \begin{equation*} [p,[g,z]] = [pk,k^{-1}\cdot[g,z]] = [pk,[k^{-1}g,z]], \end{equation*} and hence \begin{equation*} f([pk,[k^{-1}g,z]]) = [pkk^{-1}g,z] = [pg,z]. \end{equation*} It is also well defined with respect to the $G_x$-quotient. Indeed, if \begin{equation*} [g,z] = [gh,h^{-1}z], \qquad h\in G_x, \end{equation*} then \begin{equation*} f([p,[gh,h^{-1}z]]) = [pgh,h^{-1}z] = [pg,z] \end{equation*} in $E_G\times_{G_x}S_x$. Define \begin{equation*} \begin{aligned} f^{-1}: E_G\times_{G_x}S_x &\longrightarrow E_G\times_G(G\times_{G_x}S_x),\\ [p,z] &\longmapsto [p,[e,z]]. \end{aligned} \end{equation*} This map is well defined. For $h\in G_x$, \begin{align*} f^{-1}([ph,h^{-1}z]) &= [ph,[e,h^{-1}z]]\\ &= [ph,[h^{-1},z]]\\ &= [p,[e,z]]. \end{align*} Moreover, \begin{equation*} f\circ f^{-1} = \operatorname{id}_{E_G\times_{G_x}S_x} \end{equation*} and \begin{equation*} f^{-1}\circ f = \operatorname{id}_{E_G\times_G(G\times_{G_x}S_x)}. \end{equation*} Therefore $f$ is a homeomorphism. Consequently, \begin{equation*} H_G^*(G\times_{G_x}S_x;\mathbb R) \cong H_{G_x}^*(S_x;\mathbb R). \end{equation*}

Let $x_1,x_2\in\pi^{-1}(y)$ with $g\cdot x_1=x_2$, and let $G_{x_i}$ denote the isotropy group of $x_i$. Then, for $h_1\in G_{x_1}$,
\begin{equation*}
	gh_1g^{-1}\cdot x_2=gh_1g^{-1}g\cdot x_1=gh_1\cdot x_1=g\cdot x_1=x_2,
\end{equation*}
hence the isotropy group of $x_2$ is $gG_{x_1}g^{-1}$. Note that the slices $S_{x_1}$ and $S_{x_2}$ are homeomorphic and the quotient spaces $S_{x_i}/G_{x_i}$ and $\pi^{-1}(U_y)/G$ are homeomorphic.

The preceding homeomorphism suggests that the local equivariant
cohomology of a proper \(G\)-action is determined by the isotropy
action on a slice. A chain-level version of this statement will be
proved in Theorem~\ref{thm:slice-reduction-quasi-isomorphism}. In
particular, for a sufficiently small \(G_x\)-invariant slice \(S_x\),
we will obtain
\[
H_G^*(\pi^{-1}(U_y);\mathbb R)
\cong
H_{G_x}^*(S_x;\mathbb R).
\]

Let $M$ be a manifold with an action of a compact Lie group $G$. In addition to equivariant cohomology, one may also consider other cohomology theories such as $H^*(M/G)$, $H^*(\Omega(M)_{\mathrm{bas},G})$.  The relations between these cohomology theories have been studied (see \cite{Tu} for more details).

In fact, when $M\cong \mathbb{R}^n$, the cohomology of the complex $\operatorname{Weil}^{*,0}(\mathbb{R}^n,G)$	\begin{equation*}
		0\rightarrow\Omega^0(\mathbb{R}^n)_{\mathrm{bas},G}\rightarrow\Omega^1(\mathbb{R}^n)_{\mathrm{bas},G}\rightarrow\cdots
\end{equation*}
vanishes in degree $i>0$. This follows from the de Rham-type theorem for orbit spaces \cite{Ver} together with the Conner conjecture \cite{Conner}. The Conner conjecture asserts that the orbit space of any action of a compact Lie group on $\mathbb{R}^n$ is contractible. The Conner conjecture was proved by R. Oliver \cite{OLIVER}. Related results can be found in Conner \cite{Conner}, Conner and Floyd \cite{CF1,CF2}, and Floyd \cite{Fol}.

\subsection{Free Lie group actions}

Let $G$ be a compact Lie group acting smoothly and freely on a smooth manifold $M$. Since a compact group action is proper, the orbit space $M/G$ is a smooth manifold and the quotient map
\begin{equation*}
	\pi:M\longrightarrow M/G
\end{equation*}
is a principal $G$-bundle.

Let $\mathfrak g$ be the Lie algebra of $G$. Since $G$ is compact, there exists a $G$-invariant Riemannian metric on $M$. The orthogonal complement of the tangent spaces to the $G$-orbits defines a $G$-invariant horizontal distribution and hence a principal connection
\begin{equation*}
	A\in\Omega^1(M;\mathfrak g).
\end{equation*}
The connection form satisfies
\begin{equation*}
	A(X_\xi)=\xi,
	\qquad
	\varphi_g^*A=\operatorname{Ad}_{g^{-1}}A
\end{equation*}
for every $\xi\in\mathfrak g$ and $g\in G$, where $X_\xi$ is the fundamental vector field generated by $\xi$ and $\varphi_g:M\rightarrow M$ is the diffeomorphism defined by the $G$-action. Let
\begin{equation*}
	F_A=dA+\frac12[A,A]\in\Omega^2(M;\mathfrak g)
\end{equation*}
and ${\xi_1,\ldots,\xi_m}$ be a basis of $\mathfrak g$, with dual basis
${x^1,\ldots,x^m}$, write
\begin{equation*}
	A=\sum_a A^a\xi_a,
	\qquad
	F_A=\sum_a F_A^a\xi_a.
\end{equation*}
Recall that the Weil algebra $W(\mathfrak g^*)$ is generated by elements
$\theta^a$ of degree $1$ and $u^a$ of degree $2$. The connection $A$
determines the Weil homomorphism
\begin{equation*}
	w_A:W(\mathfrak g^*)\longrightarrow\Omega(M)
\end{equation*}
defined on generators by
\begin{equation*}
	w_A(\theta^a)=A^a,
	\qquad
	w_A(u^a)=F_A^a.
\end{equation*}
The structure equations for $A$ and the Bianchi identity imply that $w_A$
commutes with the differentials and is compatible with the contraction and Lie
derivative operators.

Define
\begin{equation*}
	\mathcal W_A:
	\Omega(M)\otimes W(\mathfrak g^*)
	\longrightarrow
	\Omega(M)
\end{equation*}
by
\begin{equation*}
	\mathcal W_A(\omega\otimes\sigma)
	=
	\omega\wedge w_A(\sigma).
\end{equation*}
Its restriction to the basic subcomplex gives a chain map
\begin{equation*}
	\mathcal W_A:
	\operatorname{Weil}(M,G)
	\longrightarrow
	\Omega(M)_{\mathrm{bas}}.
\end{equation*}
Since the action is free and proper, pullback by the quotient map induces an
isomorphism of differential graded algebras
\begin{equation*}
	\pi^*:\Omega(M/G)
	\xrightarrow{\cong}
	\Omega(M)_{\mathrm{bas}}.
\end{equation*}
Consequently, the composition
\begin{equation*}
	(\pi^*)^{-1}\circ\mathcal W_A:
	\operatorname{Weil}(M,G)
	\longrightarrow
	\Omega(M/G)
\end{equation*}
is a quasi-isomorphism. Therefore
\begin{equation*}
	H^*(\operatorname{Weil}(M,G))
	\cong
	H_{\mathrm{dR}}^*(M/G)
	\cong
	H_G^*(M;\mathbb R).
\end{equation*}
The induced isomorphism in cohomology is independent of the chosen connection.

Now let $U_y\subset M/G$ be a sufficiently small open neighborhood and let
$S\subset\pi^{-1}(U_y)$ be a slice. Since the action is free, the isotropy group
of every point is trivial and
\begin{equation*}
	\pi^{-1}(U_y)\cong S\times G.
\end{equation*}
Moreover, the restriction
\begin{equation*}
	\pi|_S:S\longrightarrow U_y
\end{equation*}
is a diffeomorphism. If 
\begin{equation*}
	i_S:S\hookrightarrow\pi^{-1}(U_y)
\end{equation*}
denotes the inclusion, define
\begin{equation*}
	\mathcal W_S
	=
	i_S^*\circ\mathcal W_A:
	\operatorname{Weil}(\pi^{-1}(U_y),G)
	\longrightarrow
	\Omega(S).
\end{equation*}
Since
\begin{equation*}
	\Omega(S)=\operatorname{Weil}(S,{e}),
\end{equation*}
the map $\mathcal W_S$ may equivalently be regarded as a chain map
\begin{equation*}
	\mathcal W_S:
	\operatorname{Weil}(\pi^{-1}(U_y),G)
	\longrightarrow
	\operatorname{Weil}(S,{e}).
\end{equation*}
It is a quasi-isomorphism and hence induces an isomorphism
\begin{equation*}
	H^*(\operatorname{Weil}(\pi^{-1}(U_y),G))
	\cong
	H^*(\operatorname{Weil}(S,{e}))
	=
	H_{\mathrm{dR}}^*(S).
\end{equation*}

The analogous statement for the Cartan model follows from the Mathai–Quillen
isomorphism \cite{MQ} .

When the $G$-action is not free, the quotient map is no longer a principal
$G$-bundle, and a global principal connection form satisfying
$A(X_\xi)=\xi$ cannot exist, since some fundamental vector fields may vanish.
In the next section, we construct a local reduction map adapted to the slice
model $G\times_{G_x}S_x$.

\section{Reduction over an induced \(G\)-manifold}
\label{sec:reduction-induced-space}

Let \(G\) be a compact Lie group, let \(H\subset G\) be a closed
subgroup, and let \(S\) be a smooth \(H\)-manifold. We consider the
induced \(G\)-manifold
\begin{equation*}
	M=G\times_H S,
\end{equation*}
where \(H\) acts freely on \(G\times S\) from the right by
\begin{equation}
	(g,s)\cdot h=(gh,h^{-1}\cdot s).
	\label{eq:H-action-product}
\end{equation}
The equivalence class of \((g,s)\) is denoted by \([g,s]\). The group
\(G\) acts on \(M\) from the left by
\begin{equation*}
	k\cdot[g,s]=[kg,s].
\end{equation*}

Let $\jmath:H\hookrightarrow G$ denote the inclusion homomorphism, and let
\begin{equation*}
	i:S\longrightarrow G\times_H S,
	\qquad
	i(s)=[e,s],
\end{equation*}
be the natural inclusion. We regard \(G\times_H S\) as an \(H\)-manifold
by restricting the \(G\)-action along \(\jmath\). Then \(i\) is
\(H\)-equivariant. Indeed, for every \(h\in H\),
\begin{equation*}
	i(h\cdot s)
	=[e,h\cdot s]
	=[h,s]
	=h\cdot[e,s]
	=h\cdot i(s).
\end{equation*}

This equivariant pair
\begin{equation*}
	(\jmath,i):(H,S)\longrightarrow (G,G\times_H S)
\end{equation*}
will be used to define the reduction morphisms in both the Weil and
Cartan models.

\subsection{Reduction map of the Weil and Cartan models}
\label{subsec:Weil-reduction}

We use the standard notation
\begin{equation*}
	W(\mathfrak g^*)
	=
	\Lambda(\mathfrak g^*)\otimes S(\mathfrak g^*)
\end{equation*}
for the Weil algebra of \(\mathfrak g^*\). The Lie algebra homomorphism
\[
\jmath:\mathfrak h\hookrightarrow\mathfrak g
\]
induces a morphism of differential graded algebras
\begin{equation*}
	W(\jmath):
	W(\mathfrak g^*)\longrightarrow W(\mathfrak h^*).
\end{equation*}
On generators, \(W(\jmath)\) is induced by the restriction map
\begin{equation*}
	\jmath^*:\mathfrak g^*\longrightarrow\mathfrak h^*.
\end{equation*}
More explicitly, if \(\lambda\in\mathfrak g^*\), then
\begin{equation*}
	W(\jmath)(\theta_\lambda)
	=
	\theta_{\jmath^*\lambda},
	\qquad
	W(\jmath)(u_\lambda)
	=
	u_{\jmath^*\lambda}.
\end{equation*}

Define
\begin{equation*}
	\widetilde{\operatorname{Red}}_W
	=
	i^*\otimes W(\jmath):
	\Omega(G\times_H S)\otimes W(\mathfrak g^*)
	\longrightarrow
	\Omega(S)\otimes W(\mathfrak h^*).
\end{equation*}
Thus
\begin{equation*}
	\widetilde{\operatorname{Red}}_W
	(\omega\otimes\sigma)
	=
	i^*\omega\otimes W(\jmath)(\sigma).
\end{equation*}

\begin{lemma}
	\label{lem:naturality-Weil-operators}
	For every \(\eta\in\mathfrak h\), the map
	\(\widetilde{\operatorname{Red}}_W\) satisfies
	\begin{align}
		\iota_\eta^H
		\circ\widetilde{\operatorname{Red}}_W
		&=
		\widetilde{\operatorname{Red}}_W
		\circ\iota_{\jmath(\eta)}^G,
		\label{eq:naturality-contraction}
		\\
		\mathcal L_\eta^H
		\circ\widetilde{\operatorname{Red}}_W
		&=
		\widetilde{\operatorname{Red}}_W
		\circ\mathcal L_{\jmath(\eta)}^G,
		\\
		D_H
		\circ\widetilde{\operatorname{Red}}_W
		&=
		\widetilde{\operatorname{Red}}_W
		\circ D_G.
		\label{eq:naturality-differential}
	\end{align}
\end{lemma}

\begin{proof}
	Since \(i\) is \(H\)-equivariant with respect to the inclusion
	\(\jmath:H\hookrightarrow G\), the fundamental vector fields are
	\(i\)-related:
	\begin{equation}
		T_si\bigl(X_\eta^S(s)\bigr)
		=
		X_{\jmath(\eta)}^{G\times_HS}\bigl(i(s)\bigr).
	\end{equation}
	Therefore,
	\begin{align*}
		i^*\circ\iota_{X_{\jmath(\eta)}}
		&=
		\iota_{X_\eta}\circ i^*,
		\\
		i^*\circ\mathcal L_{X_{\jmath(\eta)}}
		&=
		\mathcal L_{X_\eta}\circ i^*.
	\end{align*}
	On the Weil algebra, functoriality gives
	\begin{align*}
		\iota_\eta\circ W(\jmath)
		&=
		W(\jmath)\circ\iota_{\jmath(\eta)},
		\\
		\mathcal L_\eta\circ W(\jmath)
		&=
		W(\jmath)\circ\mathcal L_{\jmath(\eta)},
		\\
		d_{W(\mathfrak h)}\circ W(\jmath)
		&=
		W(\jmath)\circ d_{W(\mathfrak g^*)}.
	\end{align*}
	Combining these identities with the definition of the total
	contraction, Lie derivative, and differential proves
	\eqref{eq:naturality-contraction}--\eqref{eq:naturality-differential}.
\end{proof}

\begin{proposition}
	\label{prop:Weil-reduction-chain-map}
	The map \(\widetilde{\operatorname{Red}}_W\) sends \(G\)-basic elements
	to \(H\)-basic elements. Hence it restricts to a morphism of
	differential graded algebras
	\begin{equation*}
		\operatorname{Red}_W:
		\operatorname{Weil}(G\times_H S,G)
		\longrightarrow
		\operatorname{Weil}(S,H).
	\end{equation*}
\end{proposition}

\begin{proof}
	Let
	\[
	\alpha\in\operatorname{Weil}(G\times_HS,G).
	\]
	For every \(\eta\in\mathfrak h\), Lemma
	\ref{lem:naturality-Weil-operators} gives
	\begin{align*}
		\iota_\eta^H
		\bigl(\widetilde{\operatorname{Red}}_W(\alpha)\bigr)
		&=
		\widetilde{\operatorname{Red}}_W
		\bigl(\iota_{\jmath(\eta)}^G\alpha\bigr)
		=0,
		\\
		\mathcal L_\eta^H
		\bigl(\widetilde{\operatorname{Red}}_W(\alpha)\bigr)
		&=
		\widetilde{\operatorname{Red}}_W
		\bigl(\mathcal L_{\jmath(\eta)}^G\alpha\bigr)
		=0.
	\end{align*}
	Thus \(\widetilde{\operatorname{Red}}_W(\alpha)\) is \(H\)-basic.
	
	Equation \eqref{eq:naturality-differential} shows that
	\[
	D_H\circ\operatorname{Red}_W
	=
	\operatorname{Red}_W\circ D_G.
	\]
	Since both \(i^*\) and \(W(\jmath)\) preserve products,
	\(\operatorname{Red}_W\) is a morphism of differential graded
	algebras.
\end{proof}

Next, we record the functoriality of the Cartan model with respect to
equivariant pairs.  This general statement will also be used later for
the comparison morphisms between local slice models.

\begin{proposition}[Functoriality of the Cartan model]
	\label{prop:Cartan-functoriality}
	Let \(K_1\) and \(K_2\) be compact Lie groups, let \(N_1\) be a smooth
	\(K_1\)-manifold, and let \(N_2\) be a smooth \(K_2\)-manifold.
	Suppose that
	\[
	\varphi:K_1\longrightarrow K_2
	\]
	is a Lie group homomorphism and
	\[
	f:N_1\longrightarrow N_2
	\]
	is a smooth \(\varphi\)-equivariant map, namely,
	\[
	f(k\cdot x)
	=
	\varphi(k)\cdot f(x)
	\]
	for every
	\[
	k\in K_1,
	\qquad
	x\in N_1.
	\]
	
	Then the equivariant pair
	\[
	(\varphi,f):
	(K_1,N_1)
	\longrightarrow
	(K_2,N_2)
	\]
	induces a morphism of differential graded algebras
	\begin{equation}
		(\varphi,f)^*:
		\operatorname{Car}(N_2,K_2)
		\longrightarrow
		\operatorname{Car}(N_1,K_1)
	\end{equation}
	defined by
	\begin{equation}
		\bigl((\varphi,f)^*\alpha\bigr)(\xi)
		=
		f^*
		\bigl(
		\alpha(d\varphi(\xi))
		\bigr),
		\qquad
		\xi\in\mathfrak k_1.
	\end{equation}
	
	In particular,
	\[
	d_{C,K_1}\circ(\varphi,f)^*
	=
	(\varphi,f)^*\circ d_{C,K_2}.
	\]
\end{proposition}

\begin{proof}
	An element
	$\alpha\in\operatorname{Car}(N_2,K_2)$
	may be regarded as a \(K_2\)-equivariant polynomial map
	\[
	\alpha:
	\mathfrak k_2
	\longrightarrow
	\Omega(N_2).
	\]
	
	We first show that
$
	(\varphi,f)^*\alpha
	$
	belongs to
	$
	\operatorname{Car}(N_1,K_1).
	$
	Since \(f\) is \(\varphi\)-equivariant, for every \(k\in K_1\) one has
	\[
	f\circ k
	=
	\varphi(k)\circ f.
	\]
	Moreover,
	\[
	d\varphi(\operatorname{Ad}_k\xi)
	=
	\operatorname{Ad}_{\varphi(k)}
	d\varphi(\xi).
	\]
	Together with the \(K_2\)-equivariance of \(\alpha\), these identities
	show that
	$
	(\varphi,f)^*\alpha
	$
	is \(K_1\)-equivariant. Hence
	\[
	(\varphi,f)^*\alpha
	\in
	\operatorname{Car}(N_1,K_1).
	\]
	
	We next check compatibility with the Cartan differential.  Since \(f\)
	is \(\varphi\)-equivariant, the fundamental vector fields are
	\(f\)-related:
	\[
	T_xf
	\left(
	X_\xi^{N_1}(x)
	\right)
	=
	X_{d\varphi(\xi)}^{N_2}
	\left(
	f(x)
	\right).
	\]
	Therefore
	\[
	\iota_{X_\xi^{N_1}}\circ f^*
	=
	f^*\circ
	\iota_{X_{d\varphi(\xi)}^{N_2}}.
	\]
	
	For
	$
	\xi\in\mathfrak k_1,
	$
	we have
	\begin{align*}
		\bigl(
		d_{C,K_1}(\varphi,f)^*\alpha
		\bigr)(\xi)
		&=
		d
		\left(
		f^*
		\alpha(d\varphi(\xi))
		\right)
		-
		\iota_{X_\xi^{N_1}}
		\left(
		f^*
		\alpha(d\varphi(\xi))
		\right)
		\\
		&=
		f^*
		\left(
		d\alpha(d\varphi(\xi))
		\right)
		-
		f^*
		\left(
		\iota_{X_{d\varphi(\xi)}^{N_2}}
		\alpha(d\varphi(\xi))
		\right)
		\\
		&=
		f^*
		\left(
		(d_{C,K_2}\alpha)(d\varphi(\xi))
		\right)
		\\
		&=
		\bigl(
		(\varphi,f)^*d_{C,K_2}\alpha
		\bigr)(\xi).
	\end{align*}
	Thus
	\[
	d_{C,K_1}\circ(\varphi,f)^*
	=
	(\varphi,f)^*\circ d_{C,K_2}.
	\]
	
	Finally, pullback of differential forms and restriction of polynomial
	functions both preserve products. Therefore
$
	(\varphi,f)^*
	$
	is a morphism of differential graded algebras.
\end{proof}

We now apply Proposition~\ref{prop:Cartan-functoriality} to the
equivariant pair
\[
(\jmath,i):
(H,S)
\longrightarrow
(G,G\times_HS)
\]
introduced above. An element of 
$
\operatorname{Car}(G\times_HS,G)
$
may be regarded as a \(G\)-equivariant polynomial map
\[
\alpha:
\mathfrak g
\longrightarrow
\Omega(G\times_HS).
\]

Define the Cartan reduction morphism
\begin{equation}
	\operatorname{Red}_C:
	\operatorname{Car}(G\times_HS,G)
	\longrightarrow
	\operatorname{Car}(S,H)
\end{equation}
by
\begin{equation}
	\bigl(\operatorname{Red}_C\alpha\bigr)(\eta)
	=
	i^*
	\bigl(
	\alpha(d\jmath(\eta))
	\bigr),
	\qquad
	\eta\in\mathfrak h.
\end{equation}

\begin{proposition}
	\label{prop:Cartan-reduction-chain-map}
	The Cartan reduction morphism
	\[
	\operatorname{Red}_C:
	\operatorname{Car}(G\times_HS,G)
	\longrightarrow
	\operatorname{Car}(S,H)
	\]
	is a morphism of differential graded algebras.  In particular,
	\begin{equation}
		d_{C,H}\circ\operatorname{Red}_C
		=
		\operatorname{Red}_C\circ d_{C,G}.
	\end{equation}
\end{proposition}

\begin{proof}
	The inclusion
	\[
	\jmath:H\hookrightarrow G
	\]
	and the map
	\[
	i:S\longrightarrow G\times_HS,
	\qquad
	i(s)=[e,s],
	\]
	form an equivariant pair
	\[
	(\jmath,i):
	(H,S)
	\longrightarrow
	(G,G\times_HS).
	\]
	Therefore Proposition~\ref{prop:Cartan-functoriality} gives a morphism
	of differential graded algebras
	\[
	(\jmath,i)^*:
	\operatorname{Car}(G\times_HS,G)
	\longrightarrow
	\operatorname{Car}(S,H).
	\]
	By definition,
	\[
	(\jmath,i)^*
	=
	\operatorname{Red}_C.
	\]
	Hence
	$
	\operatorname{Red}_C
	$
	is a morphism of differential graded algebras.
\end{proof}

Denote the  Mathai--Quillen isomorphisms by
\begin{equation*}
	\Phi_G:
	\operatorname{Weil}(G\times_HS,G)
	\xrightarrow{\cong}
	\operatorname{Car}(G\times_HS,G)
\end{equation*}
and
\begin{equation*}
	\Phi_H:
	\operatorname{Weil}(S,H)
	\xrightarrow{\cong}
	\operatorname{Car}(S,H).
\end{equation*}

\begin{proposition}
	\label{prop:compatibility-MQ-reduction}
	The Weil and Cartan reduction morphisms are compatible with the
	Mathai--Quillen isomorphisms. More precisely, the following diagram
	commutes:
	\begin{equation*}
		\begin{CD}
			\operatorname{Weil}(G\times_H S,G)
			@>{\Phi_G}>>
			\operatorname{Car}(G\times_H S,G)
			\\
			@V{\operatorname{Red}_W}VV
			@VV{\operatorname{Red}_C}V
			\\
			\operatorname{Weil}(S,H)
			@>{\Phi_H}>>
			\operatorname{Car}(S,H).
		\end{CD}
	\end{equation*}
	Equivalently,
	\begin{equation}
		\Phi_H\circ\operatorname{Red}_W
		=
		\operatorname{Red}_C\circ\Phi_G.
	\end{equation}
	In particular,
	\begin{equation}
		\operatorname{Red}_W
		=
		\Phi_H^{-1}
		\circ\operatorname{Red}_C
		\circ\Phi_G.
	\end{equation}
\end{proposition}

\begin{proof}
	Let
	$
	m=\dim G$
	and
	$r=\dim H.
	$
	Choose a basis
	$
	\{\xi_1,\ldots,\xi_r\}
	$
	of \(\mathfrak h\), and extend it to a basis
	$
	\{\xi_1,\ldots,\xi_r,\xi_{r+1},\ldots,\xi_m\}
	$
	of \(\mathfrak g\). Let
	$
	\{x^1,\ldots,x^m\}
	$
	be the dual basis of \(\mathfrak g^*\), and let
	$
	\{\bar x^1,\ldots,\bar x^r\}
	$
	be the dual basis of \(\mathfrak h^*\). Then
	\[
	\jmath^*x^a
	=
	\begin{cases}
		\bar x^a, & 1\leq a\leq r,\\
		0, & r<a\leq m.
	\end{cases}
	\]
	Accordingly, for the degree-one generators of the Weil algebras,
	\[
	W(\jmath)(\theta_G^a)
	=
	\begin{cases}
		\theta_H^a, & 1\leq a\leq r,\\
		0, & r<a\leq m.
	\end{cases}
	\]
	
	Recall that the Mathai--Quillen transformations are
	\[
	\Phi_G=e^{\gamma_G},
	\qquad
	\Phi_H=e^{\gamma_H},
	\]
	where
	\[
	\gamma_G(\omega\otimes\sigma)
	=
	\sum_{a=1}^{m}
	\iota_{X_{\xi_a}^{G\times_HS}}\omega
	\otimes
	\theta_G^a\sigma
	\]
	and
	\[
	\gamma_H(\beta\otimes\tau)
	=
	\sum_{a=1}^{r}
	\iota_{X_{\xi_a}^{S}}\beta
	\otimes
	\theta_H^a\tau.
	\]
	
	We first show that
	\begin{equation}
		\widetilde{\operatorname{Red}}_W
		\circ\gamma_G
		=
		\gamma_H
		\circ\widetilde{\operatorname{Red}}_W.
		\label{eq:gamma-reduction-compatible}
	\end{equation}
	Indeed, for
	\(\omega\otimes\sigma
	\in
	\Omega(G\times_HS)\otimes W(\mathfrak g^*)\),
	we have
	\begin{align*}
		&
		\widetilde{\operatorname{Red}}_W
		\bigl(
		\gamma_G(\omega\otimes\sigma)
		\bigr)
		\\
		=&
		\sum_{a=1}^{m}
		i^*
		\left(
		\iota_{X_{\xi_a}^{G\times_HS}}\omega
		\right)
		\otimes
		W(\jmath)(\theta_G^a)
		W(\jmath)(\sigma)
		\\
		=&
		\sum_{a=1}^{r}
		i^*
		\left(
		\iota_{X_{\xi_a}^{G\times_HS}}\omega
		\right)
		\otimes
		\theta_H^a
		W(\jmath)(\sigma).
	\end{align*}
	Since \(i:S\to G\times_HS\) is \(H\)-equivariant, for
	\(1\leq a\leq r\) the corresponding fundamental vector fields are
	\(i\)-related. Hence
	\[
	i^*
	\circ
	\iota_{X_{\xi_a}^{G\times_HS}}
	=
	\iota_{X_{\xi_a}^{S}}
	\circ i^*.
	\]
	Therefore
	\begin{align*}
		&
		\widetilde{\operatorname{Red}}_W
		\bigl(
		\gamma_G(\omega\otimes\sigma)
		\bigr)
		\\
		=&
		\sum_{a=1}^{r}
		\iota_{X_{\xi_a}^{S}}i^*\omega
		\otimes
		\theta_H^a W(\jmath)(\sigma)
		\\
		=&
		\gamma_H
		\bigl(
		i^*\omega\otimes W(\jmath)(\sigma)
		\bigr)
		\\
		=&
		\gamma_H
		\bigl(
		\widetilde{\operatorname{Red}}_W
		(\omega\otimes\sigma)
		\bigr).
	\end{align*}
	This proves \eqref{eq:gamma-reduction-compatible}.
	
	It follows by induction that, for every \(k\geq0\),
	\[
	\widetilde{\operatorname{Red}}_W
	\circ\gamma_G^k
	=
	\gamma_H^k
	\circ\widetilde{\operatorname{Red}}_W.
	\]
	Hence
	\begin{align*}
		\widetilde{\operatorname{Red}}_W\circ\Phi_G
		&=
		\widetilde{\operatorname{Red}}_W
		\circ e^{\gamma_G}
		\\
		&=
		e^{\gamma_H}
		\circ\widetilde{\operatorname{Red}}_W
		\\
		&=
		\Phi_H
		\circ\widetilde{\operatorname{Red}}_W.
	\end{align*}
	
	On the Weil basic subcomplex,
	\(\widetilde{\operatorname{Red}}_W\) restricts to
	\(\operatorname{Red}_W\). On the Cartan side,
	\(W(\jmath)\) restricts on the symmetric algebra to the map
	\[
	S(\jmath^*):
	S(\mathfrak g^*)
	\longrightarrow
	S(\mathfrak h^*),
	\]
	and therefore the restriction of
	\(i^*\otimes W(\jmath)\) is precisely the Cartan reduction morphism
	\(\operatorname{Red}_C\). Consequently,
	\[
	\operatorname{Red}_C\circ\Phi_G
	=
	\Phi_H\circ\operatorname{Red}_W,
	\]
	which proves the proposition.
\end{proof}

\subsection{The chain quasi-isomorphism theorem}
\label{subsec:reduction-quasi-isomorphism}

Now we are ready to prove that the reduction maps induces isomorphism on cohomology.

\begin{theorem}[Equivariant induction]
	\label{thm:induction-quasi-isomorphism}
	Let \(G\) be a compact Lie group, let \(H\subset G\) be a closed
	subgroup, and let \(S\) be a smooth \(H\)-manifold. Then the reduction
	morphisms
	\begin{align}
		\operatorname{Red}_W:
		\operatorname{Weil}(G\times_HS,G)
		&\longrightarrow
		\operatorname{Weil}(S,H),
		\\
		\operatorname{Red}_C:
		\operatorname{Car}(G\times_HS,G)
		&\longrightarrow
		\operatorname{Car}(S,H)
	\end{align}
	are quasi-isomorphisms. Consequently,
	\begin{equation*}
		H_G^*(G\times_HS;\mathbb R)
		\cong
		H_H^*(S;\mathbb R).
	\end{equation*}
\end{theorem}

\begin{proof}
Firstly, let $E_G$ be a contractible free right $G$-space. By restricting the right action along the inclusion \begin{equation*} \jmath:H\hookrightarrow G, \end{equation*} the same space $E_G$ becomes a contractible free right $H$-space and hence may also be used as a model for $E_H$. Recall that \begin{equation*} M=G\times_HS, \end{equation*} where $H$ acts freely from the right on $G\times S$ by \begin{equation*} (g,s)\cdot h = (gh,h^{-1}\cdot s). \end{equation*} The Borel constructions are \begin{equation*} E_G\times_GM \qquad\text{and}\qquad E_G\times_HS, \end{equation*} where the quotient relations are \begin{equation*} [p,m] = [pg,g^{-1}\cdot m], \qquad g\in G, \end{equation*} and \begin{equation*} [p,s] = [ph,h^{-1}\cdot s], \qquad h\in H, \end{equation*} respectively. The equivariant pair \begin{equation*} (\jmath,i):(H,S)\longrightarrow(G,M), \qquad i(s)=[e,s], \end{equation*} induces a continuous map \begin{equation*} \begin{aligned} \bar i: E_G\times_HS &\longrightarrow E_G\times_GM,\\ [p,s] &\longmapsto [p,[e,s]]. \end{aligned} \end{equation*} This map is well defined. Indeed, for $h\in H$, \begin{align*} \bar i([ph,h^{-1}s]) &= [ph,[e,h^{-1}s]]\\ &= [ph,[h^{-1},s]]\\ &= [p,[e,s]]. \end{align*} Conversely, define \begin{equation*} \begin{aligned} \Psi: E_G\times_GM &\longrightarrow E_G\times_HS,\\ [p,[g,s]] &\longmapsto [pg,s]. \end{aligned} \end{equation*} The map $\Psi$ is well defined. First, if \begin{equation*} [g,s] = [gh,h^{-1}s], \qquad h\in H, \end{equation*} then \begin{equation*} [pgh,h^{-1}s] = [pg,s] \end{equation*} in $E_G\times_HS$. Second, for $k\in G$, \begin{align*} \Psi([pk,k^{-1}\cdot[g,s]]) &= \Psi([pk,[k^{-1}g,s]])\\ &= [pkk^{-1}g,s]\\ &= [pg,s]. \end{align*} Thus $\Psi$ is well defined. It is immediate that \begin{equation*} \Psi\circ\bar i = \operatorname{id}_{E_G\times_HS} \end{equation*} and \begin{equation*} \bar i\circ\Psi = \operatorname{id}_{E_G\times_GM}. \end{equation*} Therefore $\bar i$ is a homeomorphism. In particular, \begin{equation*} \bar i^{\,*}: H^*(E_G\times_GM;\mathbb R) \xrightarrow{\cong} H^*(E_G\times_HS;\mathbb R) \end{equation*} is an isomorphism.
	
Secondly, we will show that the following diagram commutes.
	\begin{equation*}
	\begin{CD}
		H^*(\operatorname{Car}(G\times_H S, G))
		@>{H^*(Red_\mathcal{C})}>>
		H^*(\operatorname{Car}(S,H))
		\\
		@V{DR_G}VV{}
		@VV{DR_H}V{}
		\\
		H^*(E_G\times_G(G\times_H S))
		@>{\bar{i}^*}>>
		H^*(E_G\times_H S)
	\end{CD}
\end{equation*}

Let \(K_1\) and \(K_2\) be compact Lie groups, let \(N_1\) be a smooth
\(K_1\)-manifold, and let \(N_2\) be a smooth \(K_2\)-manifold.
Suppose that
\[
\varphi:K_1\longrightarrow K_2
\]
is a Lie group homomorphism and
\[
f:N_1\longrightarrow N_2
\]
is a smooth \(\varphi\)-equivariant map.

By Proposition~\ref{prop:Cartan-functoriality}, the equivariant pair
\[
(\varphi,f):
(K_1,N_1)
\longrightarrow
(K_2,N_2)
\]
induces a morphism of Cartan complexes
\[
(\varphi,f)^*:
\operatorname{Car}(N_2,K_2)
\longrightarrow
\operatorname{Car}(N_1,K_1)
\]
given by
\[
\bigl((\varphi,f)^*\alpha\bigr)(\xi)
=
f^*
\bigl(
\alpha(d\varphi(\xi))
\bigr),
\qquad
\xi\in\mathfrak k_1.
\]

Choose models $E_{K_1}$ and $E_{K_2}$ as contractible free right $K_1$- and $K_2$-spaces, respectively. Let \begin{equation*} E\varphi:E_{K_1}\longrightarrow E_{K_2} \end{equation*} be a map equivariant with respect to the right actions, so that \begin{equation*} E\varphi(pk) = E\varphi(p)\varphi(k), \qquad p\in E_{K_1},\quad k\in K_1. \end{equation*} For the left $K_i$-manifolds $N_i$, the corresponding Borel constructions are defined by \begin{equation*} E_{K_i}\times_{K_i}N_i = (E_{K_i}\times N_i)/K_i, \end{equation*} where \begin{equation*} (p,x)\cdot k = (pk,k^{-1}\cdot x). \end{equation*} Since $f:N_1\rightarrow N_2$ is $\varphi$-equivariant, the pair $(E\varphi,f)$ induces a map \begin{equation*} \bar f: E_{K_1}\times_{K_1}N_1 \longrightarrow E_{K_2}\times_{K_2}N_2 \end{equation*} defined by \begin{equation*} \bar f([p,x]) = [E\varphi(p),f(x)]. \end{equation*} Indeed, if \begin{equation*} [p,x] = [pk,k^{-1}x], \end{equation*} then \begin{align*} \bar f([pk,k^{-1}x]) &= [E\varphi(pk),f(k^{-1}x)]\\ &= [E\varphi(p)\varphi(k), \varphi(k)^{-1}f(x)]\\ &= [E\varphi(p),f(x)]. \end{align*} Thus $\bar f$ is well defined.

	Let
	\[
	\mathcal{DR}_{K,N}:
	\operatorname{Car}(N,K)
	\longrightarrow
	\Omega(EK\times_KN)
	\]
	denote a chain-level equivariant de Rham morphism constructed from a
	universal connection. The functoriality of the Cartan model and the naturality of the equivariant de Rham isomorphism (see, for example, 
	\cite{Gui1} Chapter 4) gives a chain homotopy
	\begin{equation}
		\bar f^*\circ\mathcal{DR}_{K_2,N_2}
		\simeq
		\mathcal{DR}_{K_1,N_1}\circ(\varphi,f)^*.
		\label{eq:equivariant-de-Rham-chain-naturality}
	\end{equation}
The Cartan equivariant de Rham morphism is natural up to chain
homotopy with respect to equivariant maps and homomorphisms of
compact Lie groups. Therefore
\[
\bar f^*\circ\mathcal{DR}_{K_2,N_2}
\simeq
\mathcal{DR}_{K_1,N_1}\circ(\varphi,f)^* .
\]
	
	Passing to cohomology, the chain homotopy in
	\eqref{eq:equivariant-de-Rham-chain-naturality} gives the identity
	\begin{equation}
		\operatorname{DR}_{K_1,N_1}
		\circ
		H^*((\varphi,f)^*)
		=
		\bar f^{\,*}
		\circ
		\operatorname{DR}_{K_2,N_2},
		\label{eq:equivariant-de-Rham-cohomology-naturality}
	\end{equation}
	where
	\[
	\operatorname{DR}_{K,N}:
	H^*(\operatorname{Car}(N,K))
	\xrightarrow{\cong}
	H^*(EK\times_KN;\mathbb R)
	\]
	is the equivariant de Rham isomorphism.
	
	We now apply this to
	\[
	K_1=H,
	\qquad
	K_2=G,
	\qquad
	N_1=S,
	\qquad
	N_2=M=G\times_HS,
	\]
	with
	\[
	\varphi=\jmath:H\hookrightarrow G
	\]
	and
	\[
	f=i:S\longrightarrow M,
	\qquad
	i(s)=[e,s].
	\]
	The corresponding Cartan pullback is
	\[
	\bigl((\jmath,i)^*\alpha\bigr)(\eta)
	=
	i^*\bigl(\alpha(d\jmath(\eta))\bigr),
	\qquad
	\eta\in\mathfrak h.
	\]
	Since \(d\jmath:\mathfrak h\hookrightarrow\mathfrak g\) is the
	inclusion, this is exactly the Cartan reduction morphism:
	\begin{equation*}
		(\jmath,i)^*
		=
		\operatorname{Red}_C.
	\end{equation*}
	
	We use \(E_G\), with its action restricted to \(H\), as a model for
	\(E_H\). The map between universal spaces associated with
	\(\jmath:H\hookrightarrow G\) can therefore be chosen to be
	\[
	E\jmath=\operatorname{id}_{E_G}.
	\]
	The induced map is then
	\[
	\bar i:
	E_G\times_HS
	\longrightarrow
	E_G\times_GM,
	\qquad
	\bar i([p,s])=[p,[e,s]].
	\]
	Substituting these data into
	\eqref{eq:equivariant-de-Rham-cohomology-naturality}, we obtain
	\begin{equation}
		\operatorname{DR}_{H,S}
		\circ
		H^*(\operatorname{Red}_C)
		=
		\bar i^{\,*}
		\circ
		\operatorname{DR}_{G,M}.
		\label{eq:equivariant-de-Rham-reduction-naturality}
	\end{equation}
	
	Equivalently, for every class
	\[
	[\alpha]\in
	H^*(\operatorname{Car}(M,G)),
	\]
	one has
	\begin{equation*}
		\operatorname{DR}_{H,S}
		\bigl(
		[\operatorname{Red}_C(\alpha)]
		\bigr)
		=
		\bar i^{\,*}
		\left(
		\operatorname{DR}_{G,M}([\alpha])
		\right).
	\end{equation*}
	
	Since \(\bar i\) is a homeomorphism, the homomorphism
	\[
	\bar i^{\,*}:
	H^*(E_G\times_GM;\mathbb R)
	\longrightarrow
	H^*(E_G\times_HS;\mathbb R)
	\]
	is an isomorphism. Since the equivariant de Rham maps
	\(\operatorname{DR}_{G,M}\) and \(\operatorname{DR}_{H,S}\) are also
	isomorphisms, equation
	\eqref{eq:equivariant-de-Rham-reduction-naturality} gives
	\begin{equation*}
		H^*(\operatorname{Red}_C)
		=
		\operatorname{DR}_{H,S}^{-1}
		\circ
		\bar i^{\,*}
		\circ
		\operatorname{DR}_{G,M}.
	\end{equation*}
	Therefore \(H^*(\operatorname{Red}_C)\) is an isomorphism, and hence
	\(\operatorname{Red}_C\) is a quasi-isomorphism.

	At last, we apply Mathai--Quillen isomorphisms to conclude that the reduction map between the Weil model is also an isomorphism. Let
	\begin{equation*}
		\Phi_G:
		\operatorname{Weil}(M,G)
		\xrightarrow{\cong}
		\operatorname{Car}(M,G)
	\end{equation*}
	and
	\begin{equation*}
		\Phi_H:
		\operatorname{Weil}(S,H)
		\xrightarrow{\cong}
		\operatorname{Car}(S,H)
	\end{equation*}
	be the Mathai--Quillen isomorphisms.
	
	By the the Mathai--Quillen transformation with respect
	to the equivariant pair \((\jmath,i)\), the following diagram commutes:
	\begin{equation*}
		\begin{CD}
			\operatorname{Weil}(M,G)
			@>{\Phi_G}>>
			\operatorname{Car}(M,G)
			\\
			@V{\operatorname{Red}_W}VV
			@VV{\operatorname{Red}_C}V
			\\
			\operatorname{Weil}(S,H)
			@>{\Phi_H}>>
			\operatorname{Car}(S,H).
		\end{CD}
	\end{equation*}
	Thus
	\begin{equation*}
		\Phi_H\circ\operatorname{Red}_W
		=
		\operatorname{Red}_C\circ\Phi_G.
	\end{equation*}
	
	Passing to cohomology gives
	\begin{equation*}
		H^*(\operatorname{Red}_W)
		=
		H^*(\Phi_H)^{-1}
		\circ
		H^*(\operatorname{Red}_C)
		\circ
		H^*(\Phi_G).
	\end{equation*}
	The maps \(H^*(\Phi_G)\) and \(H^*(\Phi_H)\) are isomorphisms, and
	\(H^*(\operatorname{Red}_C)\) is an isomorphism.
	Consequently, $H^*(\operatorname{Red}_W)$ is also an isomorphism. Hence \(\operatorname{Red}_W\) is a
	quasi-isomorphism. This implies
	\[
	H_G^*(G\times_HS;\mathbb R)
	\cong
	H_H^*(S;\mathbb R).
	\]
\end{proof}

We next consider examples of the reduction morphism.

\begin{example}[Torus action]
	Let
	\[
	G=T^2=S^1\times S^1
	\]
	and let
	\[
	H=S^1\times\{1\}\subset T^2 .
	\]
	Then
	\[
	G/H\simeq S^1 .
	\]
	
	Let
	\[
	\mathfrak g
	=
	\mathbb{R}e_1\oplus\mathbb{R}e_2
	\]
	and assume that
	\[
	\mathfrak h=\mathbb{R}e_1 .
	\]
	Let
	\[
	\{x_1,x_2\}
	\]
	be the dual basis of \(\mathfrak g^*\). Since \(G\) is abelian,
	\[
	S(\mathfrak g^*)^G=S(\mathfrak g^*)
	=\mathbb{R}[x_1,x_2].
	\]
	
	Consider the Cartan model of the homogeneous space
	\[
	G/H.
	\]
	Since the \(G\)-action on \(G/H\) is transitive and the stabilizer of
	the identity coset is \(H\), Theorem~\ref{thm:induction-quasi-isomorphism}
	gives
	\[
	H_G^*(G/H;\mathbb R)
	\cong
	H_H^*(\mathrm{pt};\mathbb R).
	\]
	
	The Cartan reduction map is induced by the inclusion
	\[
	\jmath:\mathfrak h\hookrightarrow\mathfrak g .
	\]
	Therefore, on polynomial variables, it is the restriction map
	\[
	S(\mathfrak g^*)^G
	\longrightarrow
	S(\mathfrak h^*)^H .
	\]
	
	Let \(u\in\mathfrak h^*\) be the dual basis element of \(e_1\).
	Then
	\[
	x_1|_{\mathfrak h}=u,
	\qquad
	x_2|_{\mathfrak h}=0 .
	\]
	Hence the Cartan reduction map satisfies
	\[
	\operatorname{Red}_C(x_1)=u,
	\qquad
	\operatorname{Red}_C(x_2)=0 .
	\]
	
	Consequently, the induced homomorphism on invariant polynomial
	algebras is
	\[
	\mathbb R[x_1,x_2]
	\longrightarrow
	\mathbb R[u],
	\]
	given by
	\[
	x_1\longmapsto u,
	\qquad
	x_2\longmapsto0 .
	\]
	
	Therefore the induction theorem gives the explicit computation
	\[
	H_{T^2}^*(T^2/H;\mathbb R)
	\cong
	H_{S^1}^*(\mathrm{pt};\mathbb R)
	=
	\mathbb R[u],
	\qquad |u|=2 .
	\]
	
	This example shows that for abelian groups the Cartan reduction
	is simply the restriction of polynomial variables along the Lie
	algebra inclusion
	\[
	\mathfrak h\hookrightarrow\mathfrak g .
	\]
\end{example}

\begin{example}[The homogeneous space $SO(3)/SO(2)$]
	Let
	\[
	G=SO(3),
	\qquad
	H=SO(2)\subset SO(3),
	\]
	where \(H\) is the subgroup fixing the north pole of \(S^2\).
	Then
	\[
	G/H\simeq S^2 .
	\]
	
	By Theorem~\ref{thm:induction-quasi-isomorphism}, the reduction maps
	\[
	\operatorname{Red}_W:
	\operatorname{Weil}(S^2,SO(3))
	\longrightarrow
	\operatorname{Weil}(\mathrm{pt},SO(2))
	\]
	and
	\[
	\operatorname{Red}_C:
	\operatorname{Car}(S^2,SO(3))
	\longrightarrow
	\operatorname{Car}(\mathrm{pt},SO(2))
	\]
	are quasi-isomorphisms. Hence
	\[
	H_{SO(3)}^*(S^2;\mathbb R)
	\cong
	H_{SO(2)}^*(\mathrm{pt};\mathbb R).
	\]
	
	Since \(SO(2)\) is abelian, its Cartan model at a point is
	\[
	\operatorname{Car}(\mathrm{pt},SO(2))
	=
	S(\mathfrak{so}(2)^*).
	\]
	Let
	\[
	u\in\mathfrak{so}(2)^*
	\]
	be the standard generator. It has cohomological degree two, and
	therefore
	\[
	H_{SO(2)}^*(\mathrm{pt};\mathbb R)
	=
	\mathbb R[u],
	\qquad |u|=2 .
	\]
	
	The reduction morphism identifies the equivariant cohomology of the
	homogeneous space with the equivariant cohomology of the isotropy
	group:
	\[
	H_{SO(3)}^*(S^2;\mathbb R)
	\cong
	\mathbb R[u].
	\]
	
	More explicitly, the generator \(u\) corresponds under the induction
	isomorphism to the equivariant cohomology class on \(S^2\) obtained
	from the invariant polynomial on the Lie algebra
	$
	S(\mathfrak{so}(2)^*)^{SO(2)}.
	$
	Thus the reduction map sends the corresponding equivariant class in
	$
	H_{SO(3)}^2(S^2;\mathbb R)
	$
	to
	the universal degree-two generator of
	$
	H^2(BSO(2);\mathbb R)
	$.

	This example illustrates that for a homogeneous space \(G/H\), the
	equivariant cohomology with respect to the left \(G\)-action is
	completely determined by the characteristic classes of the isotropy
	subgroup \(H\).
\end{example}

\begin{remark}
	\label{cor:abelian-reduction}
	Assume that \(G\) is compact and Abelian, and let \(H\subset G\) be a
	closed subgroup. Then the Cartan reduction morphism
	\[
	\operatorname{Red}_C:
	\operatorname{Car}(G\times_H S,G)
	\longrightarrow
	\operatorname{Car}(S,H)
	\]
	is given by
	\[
	\bigl(\operatorname{Red}_C\alpha\bigr)(\eta)
	=
	i^*\bigl(\alpha(\eta)\bigr),
	\qquad
	\eta\in\mathfrak h.
	\]
	Equivalently, after choosing a decomposition
	\[
	\mathfrak g=\mathfrak h\oplus\mathfrak m,
	\]
	the map restricts the polynomial variables from \(\mathfrak g\) to
	\(\mathfrak h\) and pulls back the differential-form component along
	\[
	i:S\longrightarrow G\times_HS,
	\qquad
	i(s)=[e,s].
	\]
\end{remark}


\subsection{The canonical connection}
\label{subsec:canonical-connection}

Choose an \(\operatorname{Ad}_G\)-invariant inner product on
\(\mathfrak g\). Since \(G\) is compact, such an inner product exists.
Let
\begin{equation}
	\mathfrak g=\mathfrak h\oplus\mathfrak m,
	\qquad
	\mathfrak m=\mathfrak h^\perp.
	\label{eq:reductive-decomposition}
\end{equation}
The subspace \(\mathfrak m\) is \(\operatorname{Ad}_H\)-invariant.

Let
\begin{equation*}
	q:G\times S\longrightarrow G\times_H S
\end{equation*}
be the quotient map. It is a principal \(H\)-bundle with respect to
the action in \eqref{eq:H-action-product}.

Let $L_g: G\rightarrow G$ be the left translation defined by $L_g(h)=gh$, $(L_g)_*: TG\rightarrow TG$ be the tangent map, then the left Maurer–Cartan form $\theta^L\in\Omega^1(G; \mathfrak{g})$ is defined to be
\begin{equation*}
	\theta_g^L(X_g)=(L_{g^{-1}})_*X_g
\end{equation*}
for any $X_g\in T_gG$. Let $R_h:G\rightarrow G$ be the right translation, $R_h^*$ be the pull back of differential forms, then the left Maurer–Cartan form $\theta^L\in\Omega^1(G; \mathfrak{g})$ satisfies \cite{KN}
\begin{equation*}
	\begin{aligned}
		d\theta^L+\frac{1}{2}[\theta^L, \theta^L]&=0,\\
		R_h^*\theta^L&=Ad_{h^{-1}}\theta^L.			
	\end{aligned}
\end{equation*}

Let
\begin{equation*}
	\operatorname{pr}_{\mathfrak h}:
	\mathfrak g\longrightarrow\mathfrak h
\end{equation*}
be the orthogonal projection associated with
\eqref{eq:reductive-decomposition}. Define a Lie algebra valued one form $A\in\Omega^1(G\times S; \mathfrak{h})$ by
\begin{equation*}
	A_{(g,s)}(v_g,v_s)
	=
	\operatorname{pr}_{\mathfrak h}
	\bigl(\theta^L_g(v_g)\bigr).
\end{equation*}

\begin{proposition}
	\label{prop:canonical-connection}
	The form
	\[
	A\in\Omega^1(G\times S;\mathfrak h)
	\]
	is a principal connection on the principal \(H\)-bundle
	\[
	q:G\times S\longrightarrow G\times_H S.
	\]
	It satisfies
	\begin{equation*}
		A(X_{\xi})=\xi,
		\qquad
		R_h^*A=\operatorname{Ad}_{h^{-1}}A
	\end{equation*}
	for every \(\xi\in\mathfrak h\) and \(h\in H\).
\end{proposition}

\begin{proof}
	For \(\xi\in\mathfrak h\), the fundamental vector field of the right
	\(H\)-action at \((g,s)\) is
	\begin{equation*}
		X_{\xi}|_{(g,s)}
		=
		\frac{d}{dt}\bigg|_{t=0}
		\left(g\exp(t\xi),\exp(-t\xi)\cdot s\right).
	\end{equation*}
	Since
	\[
	\theta^L_g
	\left(
	\frac{d}{dt}\bigg|_{t=0}g\exp(t\xi)
	\right)
	=\xi,
	\]
	we obtain
	\[
	A(X_{\xi})=\operatorname{pr}_{\mathfrak h}(\xi)=\xi.
	\]
	Note that
	\[
	R_h^*\theta^L=\operatorname{Ad}_{h^{-1}}\theta^L.
	\]
	Because the decomposition
	\(\mathfrak g=\mathfrak h\oplus\mathfrak m\) is
	\(\operatorname{Ad}_H\)-invariant, the projection
	\(\operatorname{pr}_{\mathfrak h}\) commutes with
	\(\operatorname{Ad}_h\). Consequently,
	\[
	R_h^*A
	=
	\operatorname{pr}_{\mathfrak h}
	\left(\operatorname{Ad}_{h^{-1}}\theta^L\right)
	=
	\operatorname{Ad}_{h^{-1}}A.
	\]
	Thus \(A\) is a principal \(H\)-connection.
\end{proof}

Its curvature is
\begin{equation*}
	F_A=dA+\frac12[A,A]
	\in\Omega^2(G\times S;\mathfrak h).
\end{equation*}
If
\[
\theta^L=\theta_{\mathfrak h}+\theta_{\mathfrak m}
\]
is the decomposition determined by
\eqref{eq:reductive-decomposition}, then 
\begin{equation*}
	F_A
	=
	-\frac12
	\operatorname{pr}_{\mathfrak h}
	[\theta_{\mathfrak m},\theta_{\mathfrak m}],
\end{equation*}
implied by the Maurer--Cartan equation.

\subsection{The homogeneous-space case}
\label{subsec:homogeneous-space-case}

Taking \(S=\{\mathrm{pt}\}\), we have
\[
G\times_H\{\mathrm{pt}\}\cong G/H.
\]
Theorem \ref{thm:induction-quasi-isomorphism} immediately gives the
following result.

\begin{corollary}
	\label{cor:equivariant-cohomology-homogeneous-space}
	Let \(H\subset G\) be a closed subgroup of a compact Lie group. Then
	the reduction morphisms induce canonical isomorphisms
	\begin{equation*}
		H_G^*(G/H;\mathbb R)
		\cong
		H_H^*(\{\mathrm{pt}\};\mathbb R)
		\cong
		H^*(BH;\mathbb R).
	\end{equation*}
	In particular,
	\begin{equation}
		H_G^q(G/H;\mathbb R)
		\cong
		\begin{cases}
			0,
			& q \text{ is odd},\\[2mm]
			\bigl(S^{q/2}\mathfrak h^*\bigr)^H,
			& q \text{ is even}.
		\end{cases}
		\label{eq:cohomology-GH-explicit}
	\end{equation}
\end{corollary}

\begin{proof}
	The first isomorphism follows from Theorem
	\ref{thm:induction-quasi-isomorphism}. Since the \(H\)-action on a
	point is trivial, the Cartan complex is
	\[
	\operatorname{Car}(\{\mathrm{pt}\},H)
	=
	S(\mathfrak h^*)^H
	\]
	with zero differential, where elements of
	\(S^k(\mathfrak h^*)\) have cohomological degree \(2k\). This gives
	\eqref{eq:cohomology-GH-explicit}.
\end{proof}

\begin{example}
	Let \(G\) be a compact Lie group with \(\dim G=5\), and let
	\(H\subset G\) be a closed subgroup with \(\dim H=2\). Choose an \(\operatorname{Ad}_G\)-invariant inner product on
	\(\mathfrak g\), and let
	\[
	\mathfrak g=\mathfrak h\oplus\mathfrak m,
	\qquad
	\mathfrak m=\mathfrak h^\perp.
	\]
	Choose a basis
	$
	\{\xi_1,\xi_2\}
	$
	of \(\mathfrak h\), and extend it to a basis
	$
	\{\xi_1,\ldots,\xi_5\}
	$
	of \(\mathfrak g\), where
	$
	\{\xi_3,\xi_4,\xi_5\}
	$
	is a basis of \(\mathfrak m\). Consider the principal \(H\)-bundle
	\[
	q:G\longrightarrow G/H.
	\]
	By Proposition~\ref{prop:canonical-connection},
	\[
	A=\operatorname{pr}_{\mathfrak h}\theta^L
	\in\Omega^1(G;\mathfrak h)
	\]
	is a principal \(H\)-connection. Its curvature is
	\[
	F_A=dA+\frac12[A,A].
	\]
	
	The left \(G\)-action on \(G\) commutes with the right \(H\)-action
	and preserves \(A\). Hence \(A\) is a \(G\)-equivariant connection.
	Define
	\[
	\mu_A:\mathfrak g\longrightarrow C^\infty(G,\mathfrak h)
	\]
	by
	\[
	\mu_A(\xi)=\iota_{X_\xi}A.
	\]
	Explicitly,
	\[
	\mu_A(\xi)(g)
	=
	\operatorname{pr}_{\mathfrak h}
	\bigl(\operatorname{Ad}_{g^{-1}}\xi\bigr).
	\]
	
	Let \(d=2k\), since \(\mathfrak h\) is a two-dimensional compact Lie algebra,
	it is Abelian. Thus, in terms of the generators \(z^1,z^2\) of
	\(W(\mathfrak h^*)\), by linearity it suffices to consider a monomial 	$
	\sigma\in S^k(\mathfrak h^*)^H
	$ where
	\[
	\sigma=z^{i_1}\cdots z^{i_k},
	\qquad
	i_j\in\{1,2\},
	\]
	and provided that $\sigma$ is \(H\)-invariant.
	
	With the convention
	\[
	d_{C,G}=d-\iota_X,
	\]
	the equivariant curvature is
	\[
	F_A^G(\xi)=F_A-\mu_A(\xi).
	\]
	Define
	\[
	\beta_\sigma(\xi)
	=
	(-1)^k
	\sigma\bigl(F_A^G(\xi)\bigr).
	\]
	By the equivariant Chern--Weil construction,
	\[
	\beta_\sigma\in\operatorname{Car}^d(G/H,G)
	\]
	is \(d_{C,G}\)-closed. Let
	\[
	i:\{\mathrm{pt}\}\longrightarrow G/H,
	\qquad
	i(\mathrm{pt})=[e].
	\]
	For every \(\eta\in\mathfrak h\),
	\[
	\mu_A(\eta)(e)=\eta,
	\]
	whereas
	\[
	i^*F_A=0.
	\]
	Consequently,
	\[
	\begin{aligned}
		\bigl(\operatorname{Red}_C\beta_\sigma\bigr)(\eta)
		&=
		i^*\bigl(\beta_\sigma(\eta)\bigr)\\
		&=
		(-1)^k\sigma(-\eta)\\
		&=
		\sigma(\eta).
	\end{aligned}
	\]
	Therefore
	\[
	\operatorname{Red}_C(\beta_\sigma)=\sigma.
	\]
	Let
	\[
	\Phi_G:
	\operatorname{Weil}(G/H,G)
	\longrightarrow
	\operatorname{Car}(G/H,G)
	\]
	be the Mathai--Quillen isomorphism, and define
	\[
	\tau_\sigma=\Phi_G^{-1}(\beta_\sigma).
	\]
	Then
	\[
	\tau_\sigma\in\operatorname{Weil}^d(G/H,G)
	\]
	is a \(G\)-basic cocycle. If
	\[
	\gamma_G
	=
	\sum_{a=1}^5
	\iota_{X_{\xi_a}}\otimes\theta^a,
	\]
	then locally
	\[
	\tau_\sigma
	=
	e^{-\gamma_G}\beta_\sigma
	=
	\beta_\sigma
	-\gamma_G\beta_\sigma
	+\frac{1}{2!}\gamma_G^2\beta_\sigma
	-\cdots .
	\]
	This is a finite correction expansion in every fixed total degree.
	
	By Proposition~\ref{prop:compatibility-MQ-reduction},
	\[
	\Phi_H\circ\operatorname{Red}_W
	=
	\operatorname{Red}_C\circ\Phi_G.
	\]
	Since \(H\) is two-dimensional and hence has Abelian Lie algebra,
	the Mathai--Quillen transformation on the point identifies
	\(S(\mathfrak h^*)^H\) with the basic Weil algebra. Hence
	\[
	\operatorname{Red}_W(\tau_\sigma)
	=
	1\otimes\sigma.
	\]
	Therefore
	\[
	H^*(\operatorname{Red}_W)([\tau_\sigma])
	=
	[1\otimes\sigma],
	\]
	and
	\[
	H^*(\operatorname{Red}_W)^{-1}
	\bigl([1\otimes\sigma]\bigr)
	=
	[\tau_\sigma].
	\]
	
	When \(H\) is normal in \(G\), the construction simplifies. In
	particular, if the polynomial \(\sigma\), regarded in
	\(W(\mathfrak g^*)\), is already \(G\)-basic, then one may take
	\[
	\tau_\sigma=1\otimes\sigma.
	\]
\end{example}

\subsection{Application to a slice}
\label{subsec:application-slice}

Let \(G\) act properly on a smooth manifold \(N\), let \(x\in N\), and
let
\[
H=G_x
\]
be the isotropy group of \(x\). Let \(S_x\) be an \(H\)-invariant
slice at \(x\). After shrinking the slice if necessary, the slice
theorem gives a \(G\)-equivariant diffeomorphism
\begin{equation*}
	\varphi:
	G\times_HS_x
	\xrightarrow{\cong}
	U_x,
\end{equation*}
where \(U_x\subset N\) is a \(G\)-invariant neighborhood of the orbit
\(G\cdot x\).

Define
\begin{equation*}
	\operatorname{Red}_{W,x}
	=
	\operatorname{Red}_W\circ\varphi^*:
	\operatorname{Weil}(U_x,G)
	\longrightarrow
	\operatorname{Weil}(S_x,H),
\end{equation*}
and
\begin{equation*}
	\operatorname{Red}_{C,x}
	=
	\operatorname{Red}_C\circ\varphi^*:
	\operatorname{Car}(U_x,G)
	\longrightarrow
	\operatorname{Car}(S_x,H).
\end{equation*}

\begin{theorem}
	\label{thm:slice-reduction-quasi-isomorphism}
	The maps
	\[
	\operatorname{Red}_{W,x}
	\quad\text{and}\quad
	\operatorname{Red}_{C,x}
	\]
	are quasi-isomorphisms. Consequently,
	\begin{equation*}
		H_G^*(U_x;\mathbb R)
		\cong
		H_{G_x}^*(S_x;\mathbb R).
	\end{equation*}
\end{theorem}

\begin{proof}
	The pullback \(\varphi^*\) is an isomorphism of the corresponding
	Weil and Cartan complexes because \(\varphi\) is a \(G\)-equivariant
	diffeomorphism. The conclusion therefore follows from Theorem
	\ref{thm:induction-quasi-isomorphism}.
\end{proof}

If \(S_x\) is chosen to be an invariant open ball in the normal
representation at \(x\), then radial contraction gives a
\(G_x\)-equivariant homotopy between \(S_x\) and the point \(x\).
Therefore we obtain the following local computation.

\begin{corollary}
	\label{cor:local-equivariant-cohomology}
	Let \(S_x\) be a sufficiently small \(G_x\)-invariant ball. Then
	\begin{equation*}
		H_G^q(U_x;\mathbb R)
		\cong
		H_{G_x}^q(S_x;\mathbb R)
		\cong
		H^q(BG_x;\mathbb R).
	\end{equation*}
	Equivalently,
	\begin{equation*}
		H_G^q(U_x;\mathbb R)
		\cong
		\begin{cases}
			0,
			& q \text{ is odd},\\[2mm]
			\bigl(S^{q/2}\mathfrak g_x^*\bigr)^{G_x},
			& q \text{ is even}.
		\end{cases}
	\end{equation*}
\end{corollary}

\begin{proof}
	The radial homotopy
	\[
	r_t:S_x\longrightarrow S_x,
	\qquad
	r_t(v)=tv,
	\qquad
	0\leq t\leq1,
	\]
	is \(G_x\)-equivariant. Thus \(S_x\) is \(G_x\)-equivariantly
	contractible and
	\[
	H_{G_x}^*(S_x;\mathbb R)
	\cong
	H_{G_x}^*(\{\mathrm{pt}\};\mathbb R).
	\]
	The result follows from Theorem
	\ref{thm:slice-reduction-quasi-isomorphism}.
\end{proof}

\section{Slice groupoids}

In this section we define a special type of groupoid, called a slice groupoid. Slice groupoids are locally equivalent to action groupoids $G\ltimes S$, where $S$ is a contractible space. The notion of a slice groupoid generalizes the local framework of orbifolds and action groupoids. We also generalize equivariant cohomology to slice groupoids using sheaf-theoretic methods. 

Given two groupoids $\mathcal{G}$ and $\mathcal{H}$, a homomorphism
\begin{equation*}
	\phi:\mathcal{H}\rightarrow\mathcal{G}
\end{equation*}
consists of two continuous maps
\begin{equation*}
	\begin{aligned}
		\phi_0:H_0\rightarrow G_0\\
		\phi_1:H_1\rightarrow G_1
	\end{aligned}
\end{equation*}
that commute with all the structure maps of the groupoids. If both $\phi_0$ and $\phi_1$ are embeddings, we call $\phi$ an embedding.

\begin{definition}
	A homomorphism $\phi:\mathcal{H}\rightarrow\mathcal{G}$ is called an equivalence if 
\begin{enumerate}
	\item the mapping
	\begin{equation*}
		t\circ\pi_1: G_1\text{\ }_s\times_\phi\text{\ }H_0\rightarrow G_0
	\end{equation*}
defined on the fibered product $\{(g,y)\text{\ }|\text{\ }g\in G_1,y\in H_0, s(g)=\phi(y)\}$ is a continuous open surjection, where $\pi_1$ is the projection to the first component.\\
\item $H_1$ is homeomorphic to the fibre product 
\begin{equation*}
	H_1\cong (H_0\times H_0)\text{\ }_{\phi_0\times\phi_0}\times_{(s,t)}\text{\ }G_1.
\end{equation*}
\end{enumerate}
\end{definition}
Note that $\mathcal{H}$ and $\mathcal{G}$ are not required to be Lie groupoids. A similar definition is given in \cite{AAR} when $\mathcal{H}$ and $\mathcal{G}$ are Lie groupoids. If $\mathcal{H}$ and $\mathcal{G}$ are equivalent, then $|\mathcal{G}|$ is homeomorphic to $|\mathcal{H}|$. 

\begin{remark}
	Given two groupoids $\mathcal{G}$ and $\mathcal{H}$, if there exist a groupoid $\mathcal{K}$ and equivalence
	\begin{equation*}
		\mathcal{G}\leftarrow\mathcal{K}\rightarrow\mathcal{H},
	\end{equation*}
then we say that $\mathcal{G}$ and $\mathcal{H}$ are Morita equivalent.
\end{remark}

\begin{definition}
	Let $\mathcal{G}:G_1\rightrightarrows G_0$ be a groupoid such that $|\mathcal{G}|$ is Hausdorff, and let $\pi:G_0\rightarrow|\mathcal{G}|$ be the orbit projection. For every $y\in|\mathcal{G}|$, suppose that there exist a neighborhood $U_y$ and an equivalence 
	\begin{equation*}
		\Psi_{U_y}:\mathcal{G}|_{U_y}\rightarrow G_y\ltimes S_y
	\end{equation*}
	where $S_y$ is a space diffeomorphic to a Euclidean space $\mathbb{R}^m$, all points in the orbit corresponding to $y$ are mapped to the origin $0\in S_y$, and $G_y$ fixes the origin. We call $\mathcal{G}$ locally linearizable, and we call $(U_y,\Psi_{U_y})$ a local linearized chart of $\mathcal{G}$.
\end{definition}
   
Note that the dimensions of $G_y$ and $S_y$ depend on $y$, and $\mathcal{G}$ is not required to be a Lie groupoid in the definition. Globally, a local linearizable groupoid might not be equivalent or Morita equivalent to an action groupoid.

\begin{example}
Let $G_0=\mathbb{R}\sqcup\mathbb{C}$. We construct the groupoid by gluing a $\mathbb{Z}_2$-action on $\mathbb{R}$ and an $S^1$-action on $\mathbb{C}$.
\begin{center}
	\includegraphics[width=0.5\textwidth]{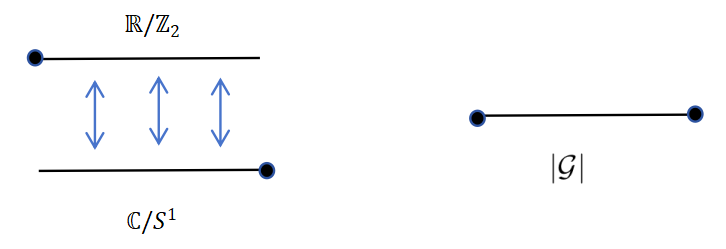}
\end{center}
Explicitly, let $A$ and $A^{-1}$ be the sets
	\begin{equation*}
		\begin{split}
					A=&\{(x,z)\in G_0\times G_0\text{\ \ }|\text{\ \ } x\in\mathbb{R},z\in\mathbb{C},x\neq 0, |z|=\frac{1}{|x|}\}\\
					A^{-1}=&\{(z,x)\in G_0\times G_0\text{\ \ }|\text{\ \ } x\in\mathbb{R},z\in\mathbb{C},x\neq 0, |z|=\frac{1}{|x|}\}. 
		\end{split}
	\end{equation*}
Let 
\begin{equation*}
	G_1=(\mathbb{Z}_2\times\mathbb{R})\sqcup (S^1\times\mathbb{C})\sqcup A\sqcup A^{-1}
\end{equation*}
and define the source and target maps by
\begin{equation*}
	\begin{aligned}
			&s((g,x))=x,\text{\ \ \ \ \ \ } &&t((g,x))=gx, &&\text{\ \ \ \ for $(g,x)\in\mathbb{Z}_2\times\mathbb{R}$}\\
		&s((e^{i\theta},z))=z,\text{\ \ \ \ } &&t((e^{i\theta},z))=e^{i\theta}z, &&\text{\ \ \ \ for $(e^{i\theta},z)\in S^1\times\mathbb{C}$}\\ 
		&s((x,z))=x,\text{\ \ \ \ \ \ } &&t((x,z))=z, &&\text{\ \ \ \ for $(x,z)\in A$}\\
		&s((z,x))=z,\text{\ \ \ \ \ \ } &&t((z,x))=x, &&\text{\ \ \ \ for $(z,x)\in A^{-1}$},
	\end{aligned}
\end{equation*}
Then one can check that $\mathcal{G}:G_1\rightrightarrows G_0$ is a groupoid. Let $|\mathcal{G}|$ be the orbit space of $\mathcal{G}$ and
\begin{equation*}
	\pi:G_0\rightarrow|\mathcal{G}|
\end{equation*}
be the projection. Then there are three types of points in $|\mathcal{G}|$:
\begin{enumerate}
	\item The first type is $\pi(x_0)$ where $x_0=0\in\mathbb{R}\subset G_0$. We can choose a neighborhood $U_{\pi(x_0)}\subset|\mathcal{G}|$ such that $\mathcal{G}|_{U_{\pi(x_0)}}$ is Morita equivalent to an action groupoid $\mathbb{Z}_2\ltimes S_0$ where $S_0$ is diffeomorphic to $\mathbb{R}$.
	\item The second type is $\pi(x)$ where $x\in\mathbb{R}\subset G_0$ and $x\neq 0$. We can choose a neighborhood $U_{\pi(x)}$ of $\pi(x)$ such that $\mathcal{G}|_{U_{\pi(x)}}$ is Morita equivalent to an action groupoid $\{e\}\ltimes S_x$ where $S_x$ is diffeomorphic to $\mathbb{R}$.
	\item The third type is $\pi(z_0)$ where $z_0=0\in\mathbb{C}\subset G_0$. We can choose a neighborhood $U_{\pi(z_0)}\subset|\mathcal{G}|$ such that $\mathcal{G}|_{U_{\pi(z_0)}}$ is Morita equivalent to an action groupoid $S^1\ltimes S_{z_0}$ where $S_{z_0}$ is diffeomorphic to $\mathbb{C}$. 
\end{enumerate}
\end{example}

\begin{definition}
	\label{def:compatible-local-linearized-charts}
	Let
	$\mathcal G:G_1\rightrightarrows G_0
	$
	be a groupoid and let
	$
	X=|\mathcal G|
	$
	be its orbit space. A local linearized chart on an open subset
	$
	U\subset X
	$
	consists of a local model
	\[
	\Psi_U:
	\mathcal G|_U
	\longrightarrow
	G_U\ltimes S_U,
	\]
	where $G_U$ is a compact Lie group and $S_U$ is a smooth
	$G_U$-manifold.
	
	We say that a collection
	\[
	\left\{
	(U,\Psi_U)
	\right\}_{U\in\mathcal B}
	\]
	of local linearized charts is \emph{compatible} if the following
	conditions are satisfied.
	
	\begin{enumerate}
		
		\item
		The collection $\mathcal B$ of chart domains forms a basis for the
		topology of $X$.  In particular, for every open subset $O\subset X$
		and every point $y\in O$, there exists $U\in\mathcal B$ such that
		\[
		y\in U\subset O.
		\]
		
		\item
		Let
		$
		(U,\Psi_U)$,
		and
		$(V,\Psi_V)
		$
		be two local linearized charts and let
		$
		y\in U\cap V.
		$
		Then there exists a local linearized chart
		$
		(W,\Psi_W)$ with
	   $y\in W\subset U\cap V,
		$
		together with embeddings of action groupoids
		\[
		\lambda_{W,U}:
		G_W\ltimes S_W
		\longrightarrow
		G_U\ltimes S_U
		\]
		and
		\[
		\lambda_{W,V}:
		G_W\ltimes S_W
		\longrightarrow
		G_V\ltimes S_V.
		\]
		
		Each such embedding is induced by an equivariant pair
		\[
		(\jmath_{W,U},i_{W,U}):
		(G_W,S_W)
		\longrightarrow
		(G_U,S_U)
		\]
		and, respectively,
		\[
		(\jmath_{W,V},i_{W,V}):
		(G_W,S_W)
		\longrightarrow
		(G_V,S_V),
		\]
		where
		\[
		\jmath_{W,U}:G_W\longrightarrow G_U,
		\qquad
		\jmath_{W,V}:G_W\longrightarrow G_V
		\]
		are Lie group homomorphisms and
		\[
		i_{W,U}:S_W\longrightarrow S_U,
		\qquad
		i_{W,V}:S_W\longrightarrow S_V
		\]
		are equivariant smooth embeddings.
		
		Let
		\[
		q_U:S_U\longrightarrow S_U/G_U\simeq U,
		\]
		\[
		q_V:S_V\longrightarrow S_V/G_V\simeq V,
		\]
		and
		\[
		q_W:S_W\longrightarrow S_W/G_W\simeq W
		\]
		be the quotient maps.  The embeddings are required to induce the
		ordinary inclusions on the orbit spaces:
		\[
		q_U\circ i_{W,U}
		=
		\iota_{W,U}\circ q_W
		\]
		\[
		q_V\circ i_{W,V}
		=
		\iota_{W,V}\circ q_W,
		\]
		where
		\[
		\iota_{W,U}:W\hookrightarrow U,
		\qquad
		\iota_{W,V}:W\hookrightarrow V
		\]
		are the inclusion maps.
		
		\item
		Refinement embeddings are compatible with composition.  More
		precisely, suppose
		$
		W\subset V\subset U
		$
		are chart domains and refinement embeddings
		\[
		\lambda_{W,V}:
		G_W\ltimes S_W
		\longrightarrow
		G_V\ltimes S_V
		\]
		\[
		\lambda_{V,U}:
		G_V\ltimes S_V
		\longrightarrow
		G_U\ltimes S_U
		\]
		are given.  Then their composition
		\[
		\lambda_{V,U}\circ\lambda_{W,V}:
		G_W\ltimes S_W
		\longrightarrow
		G_U\ltimes S_U
		\]
		is again an admissible refinement embedding.
		
		Equivalently, for the corresponding equivariant pairs,
		\[
		\jmath_{W,U}
		=
		\jmath_{V,U}\circ\jmath_{W,V},
		\]
		\[
		i_{W,U}
		=
		i_{V,U}\circ i_{W,V}
		\]
		whenever $\lambda_{W,U}$ is defined to be this composite.
		
		\item
		The refinement embeddings are locally coherent.  Namely, suppose
		two refinement embeddings
		\[
		\lambda_{V,U}^{(1)},
		\lambda_{V,U}^{(2)}:
		G_V\ltimes S_V
		\longrightarrow
		G_U\ltimes S_U
		\]
		represent two choices of comparison between the same local models.
		For every point
		$
		y\in V,
		$
		there exists a chart domain
		$
		W\in\mathcal B$,
		with
		$y\in W\subset V,
		$
		and a refinement embedding
		\[
		\lambda_{W,V}:
		G_W\ltimes S_W
		\longrightarrow
		G_V\ltimes S_V
		\]
		such that the two induced comparison morphisms agree after restriction
		to $W$. In particular, the corresponding morphisms of Cartan complexes agree
		after passing to $W$.
		
	\end{enumerate}
	
	A groupoid equipped with such a compatible system of local linearized
	charts will be called a \emph{slice groupoid}.
\end{definition}

\begin{remark}
	\label{rem:compatible-charts-local}
	The compatibility condition above is local.  It does not require a
	specified comparison morphism for every inclusion of chart domains
	$
	V\subset U.
	$
	Instead, it requires that comparison morphisms exist after passing to
	sufficiently small neighborhoods and that different choices become
	equal after a further refinement.
	
	This is precisely the amount of compatibility needed to define germs
	of local Cartan elements.
\end{remark}

Recall that when $M$ is a manifold with a proper action of a compact Lie group $G$, the slice theorem gives, for every $y\in M/G$ and every $x\in\pi^{-1}(y)$,
\begin{equation*}
	\widetilde{U}\simeq G\times_{G_x}S_x,
\end{equation*}
where $\widetilde{U}\subset M$ is a $G$-invariant neighborhood of the orbit $\pi^{-1}(y)$. We use the right $G_x$-action on $G\times S_x$ and the left $G$-action on $G\times_{G_x}S_x$ given by
\begin{equation*}
	(g,p)\cdot h=(gh,h^{-1}\cdot p)
\end{equation*}
and
\begin{equation*}
	k\cdot[g,q]=[kg,q].
\end{equation*}

\begin{theorem}
	The action groupoids $G_x\ltimes S_x$ and $G\ltimes(G\times_{G_x}S_x)$
	are equivalent. In particular, they are Morita equivalent.
\end{theorem}

\begin{proof}
	Define 
	\[
	\Phi:G_x\ltimes S_x
	\longrightarrow
	G\ltimes(G\times_{G_x}S_x)
	\]
by
\begin{equation*}
	\begin{aligned}
		\Phi_0:S_x&\longrightarrow G\times_{G_x}S_x\\
		s&\mapsto[e,s]
	\end{aligned}
\end{equation*}
and

\begin{equation*}
	\begin{aligned}
			\Phi_1:G_x\times S_x
		&\longrightarrow
		G\times(G\times_{G_x}S_x)\\
	(h,s)&\mapsto(h,[e,s]).
	\end{aligned}
\end{equation*}
Then
	\[
	s\bigl(\Phi_1(h,s)\bigr)=[e,s]=\Phi_0(s),
	\]
	while
	\[
	t\bigl(\Phi_1(h,s)\bigr)
	=
	h\cdot[e,s]
	=
	[h,s].
	\]
	Since
	\[
	[h,s]=[e,h\cdot s],
	\]
	we obtain
	\[
	t\bigl(\Phi_1(h,s)\bigr)
	=
	[e,h\cdot s]
	=
	\Phi_0(h\cdot s).
	\]
	Thus \(\Phi\) is a well-defined morphism of action groupoids.
	
	Next, we verify that \(\Phi\) is an equivalence. Firstly, consider the map
	\[
	t\circ\operatorname{pr}_1:
	\bigl(G\times(G\times_{G_x}S_x)\bigr)
	{}_s\times_{\Phi_0} S_x
	\longrightarrow
	G\times_{G_x}S_x,
	\]
	where
	\[
	\bigl(G\times(G\times_{G_x}S_x)\bigr)
	{}_s\times_{\Phi_0} S_x
	=
	\left\{
	\bigl((g,[e,s]),s\bigr)
	\,\middle|\,
	g\in G,\ s\in S_x
	\right\}.
	\]
	This fiber product is naturally diffeomorphic to \(G\times S_x\) via
	\begin{equation*}
		\begin{aligned}
				G\times S_x
			&\longrightarrow
			\bigl(G\times(G\times_{G_x}S_x)\bigr)
			{}_s\times_{\Phi_0} S_x\\
				(g,s)&\longmapsto \bigl((g,[e,s]),s\bigr).
		\end{aligned}
	\end{equation*}
	Under this identification, the map
	\(t\circ\operatorname{pr}_1\) becomes
	\[
	q:G\times S_x\longrightarrow G\times_{G_x}S_x,
	\qquad
	q(g,s)=[g,s].
	\]
	The map \(q\) is the quotient map of the free and proper right
	\(G_x\)-action
	\[
	(g,s)\cdot h=(gh,h^{-1}s).
	\]
	Hence \(q\) is a surjective submersion. Therefore the first condition
	for equivalence is satisfied.
	
	It remains to verify full faithfulness. Consider the fiber product
	\[
	(S_x\times S_x)
	{}_{\Phi_0\times\Phi_0}\times_{(s,t)}
	\bigl(G\times(G\times_{G_x}S_x)\bigr).
	\]
	An element of this fiber product is a triple
	\[
	(s_1,s_2,g)
	\]
	such that
	\[
	[e,s_1]
	=
	s(g,[e,s_1])
	\]
	and
	\[
	g\cdot[e,s_1]
	=
	[e,s_2].
	\]
	Thus the condition is precisely
	\[
	[g,s_1]=[e,s_2].
	\]
	By the definition of the quotient \(G\times_{G_x}S_x\), this equality
	holds if and only if there exists \(h\in G_x\) such that
	\[
	g=h
	\qquad\text{and}\qquad
	s_2=h\cdot s_1.
	\]
	Indeed,
	\[
	[g,s_1]=[e,s_2]
	\]
	means that for some \(h\in G_x\),
	\[
	(g,s_1)\cdot h=(e,s_2),
	\]
	that is,
	\[
	gh=e,
	\qquad
	h^{-1}s_1=s_2.
	\]
	Equivalently, with \(k=h^{-1}\in G_x\),
	\[
	g=k,
	\qquad
	s_2=k\cdot s_1.
	\]
	Hence the map
	\[
	G_x\times S_x
	\longrightarrow
	(S_x\times S_x)
	{}_{\Phi_0\times\Phi_0}\times_{(s,t)}
	\bigl(G\times(G\times_{G_x}S_x)\bigr),
	\]
	defined by
	\[
	(h,s)
	\longmapsto
	\bigl(s,h\cdot s,h\bigr),
	\]
	is a diffeomorphism. Under this identification, the canonical map
	from the arrow space of \(G_x\ltimes S_x\) to the above fiber product
	is exactly this diffeomorphism.
	
	Therefore \(\Phi\) is fully faithful and essentially surjective.
	Hence
	$
	G_x\ltimes S_x
	$
	and
	$
	G\ltimes(G\times_{G_x}S_x)
	$
	are equivalent groupoids. Consequently, they are Morita equivalent.
\end{proof}

\begin{example} \label{exa c1}
	An action groupoid $G\ltimes M$ with a proper $G$-action is a slice groupoid. One may also construct another slice groupoid that encodes the slice data. Explicitly, let $\{U_\alpha\}_\alpha$ be an open cover of the orbit space $|G\ltimes M|$ such that the slice theorem holds for each $U_\alpha$. We take a slice $S_{\alpha}$ in $M$ and define
\begin{equation*}
		G_0=\bigsqcup_\alpha S_\alpha,
\end{equation*}
	then we have a natural inclusion
	\begin{equation*}
		i:G_0\rightarrow M.
	\end{equation*}
    Let
\begin{equation*}
	G_1=\{(p,q,g)\in G_0\times G_0\times G\text{\ }|\text{\ }g\cdot i(p)=i(q) \},
\end{equation*}
and define the source and target maps by
\begin{equation*}
	\begin{aligned}
		s(p,q,g)&=p\\
		t(p,q,g)&=q.
	\end{aligned}
\end{equation*}
Then the local charts define a slice groupoid $G_1\rightrightarrows G_0$.
\end{example}

\begin{example}
	Orbifold groupoids are slice groupoids because an orbifold is locally modeled by an action groupoid $G\ltimes\mathbb{R}^n$ for a finite group $G$. 
\end{example}

We give one more example of a slice groupoid whose orbit space is not a global quotient.

\begin{example}
	Let $U=D/Z_m$ be a topological space, where $D$ is the open unit disk in $\mathbb{C}\simeq\mathbb{R}^2$ and $\mathbb{Z}_m$ is a finite group whose generator $\tau$ acts on $D$ by
	\begin{equation*}
		\tau\cdot z=e^{i\frac{2\pi}{m}}z.
	\end{equation*} 
Let $V=(D\times S^1)/S^1$ where $S^1$ acts on $D\times S^1$ by 
\begin{equation*}
	e^{i\theta}\cdot(z,e^{i\lambda})=(e^{i\theta}z,e^{i\lambda}).
\end{equation*}
The orbit space $|\mathcal{G}|$ is given by gluing $U$ and $V$ as follows
\begin{equation*}
	|\mathcal{G}|=(U\sqcup V)/\sim
\end{equation*}
where $x\in U$, $y\in V$ are equivalent if 
\begin{equation*}
	\begin{split}
		x=&[re^{i\lambda}],\text{\ \ } r\neq 0\\
		y=&[(1-r)e^{i2\pi}, e^{im\lambda}].
	\end{split}
\end{equation*}  
\begin{center}
	\includegraphics[width=0.5\textwidth]{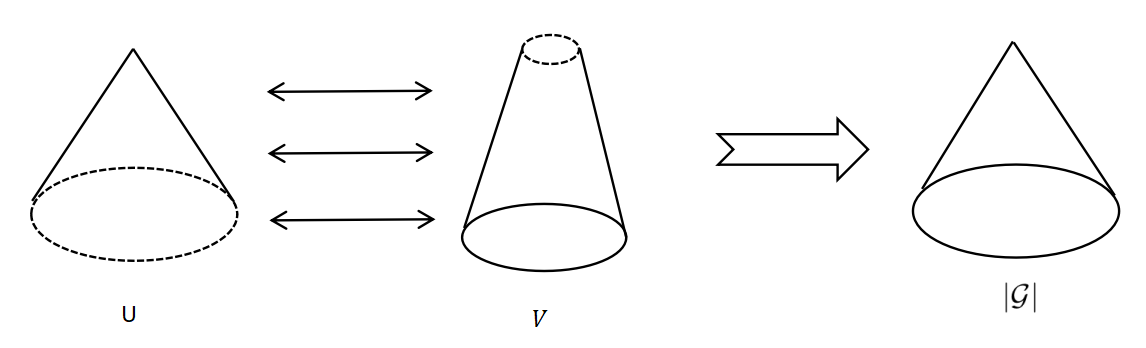}
\end{center}
Explicitly, let $\mathcal{G}=G_1\rightrightarrows G_0$ be a groupoid, where
\begin{equation*}
	G_0=D\sqcup(D\times S^1)
\end{equation*}
and
\begin{equation*}
	\begin{split}
		G_1=&(\mathbb{Z}_m\times D)\sqcup\left(S^1\times(D\times S^1)\right)\\
		&\sqcup\{(x,y)\in D\times(D\times S^1)\text{\ }\big|\text{\ }x=re^{i\theta}, \text{\ }y=(z, e^{im\theta}), \text{\ }0<r<1,\text{\ }|z|=1-r\}\\
		&\sqcup\{(y,x)\in (D\times S^1)\times D\text{\ }\big|\text{\ }x=re^{i\theta}, \text{\ }y=(z, e^{im\theta}), \text{\ }0<r<1,\text{\ }|z|=1-r\}\\
	\end{split}	
\end{equation*}
where the source and target maps on the action-groupoid components generate the group actions on $G_0$, while the source and target maps on the remaining two components generate the gluing arrows and their inverses.

In this example, there are three types of points in $|\mathcal{G}|$:
\begin{itemize}
	\item  If $p=[0,\theta]\in V\subset|\mathcal{G}|$, there is a neighborhood $V'_p\subset|\mathcal{G}|$ of $p$ such that $\mathcal{G}|_{V'_p}$ is Morita equivalent to $S^1\ltimes\left(D_\delta\times(\theta-\epsilon,\theta+\epsilon)\right)$, where $D_\delta$ is a disk with radius $\delta$, $\epsilon$ is a small real number greater than zero, and $S^1$ acts on $D_\delta$ by rotation.\\
	\item  If $q=[0]\in U\subset|\mathcal{G}|$, there is a neighborhood $U'_q\subset|\mathcal{G}|$ of $q$ such that $\mathcal{G}|_{U'_q}$ is Morita equivalent to $\mathbb{Z}_m\ltimes D_\delta$.\\
	\item  The other points are smooth points, i.e. there is a neighborhood $W'$ of each point such that $\mathcal{G}|_{W'}$ is Morita equivalent to $\{e\}\ltimes D_{\delta}$. where $\{e\}$ is the trivial group.
\end{itemize}
Therefore for any point in $|\mathcal{G}|$, there exists a neighborhood that is locally modeled on the quotient of a slice by a compact Lie group action.
\end{example} 

This example illustrates that slice groupoids encompass spaces that are not global quotients, thereby generalizing the framework of orbifolds and proper Lie group actions.


\section{Equivariant cohomology on slice groupoids}

\subsection{Sheaf theory}

In this section we use the local Cartan reduction morphisms, germs,
and the associated espace étalé to define equivariant cohomology on
slice groupoids. We shall start with sheaf theory. The main references for this section are \cite{Bry,Max,Bredon}.

\begin{definition}
Let $\mathcal{F}'$ be a presheaf of sets on a topological space $X$, and let $\mathcal{F}'_p$ be the stalk of $\mathcal{F}'$ at $p\in X$. The sheafification of $\mathcal{F}'$ is defined to be the sheaf $\mathcal{F}$ such that for each open set $U\subset X$,
\begin{align*}
\mathcal{F}(U):=&\{\varphi\in\Gamma(\bigsqcup_{p\in U}\mathcal{F}'_{p})\text{\ }\big|\text{\ } 
\text{\ for every \ } p\in U,\text{\ there is a neighborhood\ } V\subset U \text{\ of $p$ and }\\
&\text{\ a section } \varphi'\in \mathcal{F}'(V)\text{\ with }\varphi(q)=\varphi'_q\in\mathcal{F}'_q \text{\ for all } q\in V \},
\end{align*}
where $\varphi'_q$ is the germ determined by $\varphi'$.
\end{definition}

Note that $\mathcal{F}$ is indeed a sheaf, and the stalks of a presheaf $\mathcal{F}'$ and its sheafification $\mathcal{F}$ are the same at all points.\\

\begin{remark}
The disjoint union of all stalks forms a bundle-like space over $X$, called the \'{e}tal\'{e} space. The sheafification may be realized as the sheaf of continuous sections of this \'{e}tal\'{e} space.
\end{remark}

\begin{definition}
A differential graded sheaf (DGS) $\mathcal{F}^{\bullet}$ is a complex of sheaves $\mathcal{F}^i$ with morphisms $d_i:\mathcal{F}^i\rightarrow \mathcal{F}^{i+1} $ such that $d_{i+1}\circ d_i=0$. The derived sheaf (cohomology sheaf) $\mathcal{H}^{\bullet}\mathcal{F}^{\bullet}$ is the DGS associated to the presheaf $U \longmapsto H^{\bullet}(\mathcal{F}^{\bullet}(U))$ with zero differential. Its stalk at $x\in X$ is $H^{\bullet}(\mathcal{F}^{\bullet}_x)$.
\end{definition}

Thus, for each open set $U\subset X$, we obtain the cohomology groups of the complex $\mathcal{F}^{\bullet}(U)$, regarded as a graded object with zero differential. We do this for every open set to get a presheaf with the restriction map induced by the restriction maps of $\mathcal{F}^{\bullet}$. The sheafification of the presheaf is the cohomology sheaf.\\

\begin{definition}
Let $\phi:K^{\bullet}\rightarrow L^{\bullet}$ be a morphism of complexes of sheaves. If, for every $i$, $\phi$ induces an isomorphism $\mathcal{H}^i(K^{\bullet})\rightarrow \mathcal{H}^i(L^{\bullet})$ of the cohomology sheaves, then $\phi$ is called a quasi-isomorphism.
\end{definition}

\begin{definition}
Let $K^{\bullet}$ be a complex of sheaves on a space $X$ that is bounded below (i.e., its terms vanish for $i<0$). An injective resolution of $K^{\bullet}$ is a double complex of sheaves $I^{\bullet\bullet}$, such that $I^{p\bullet}$ is an injective resolution of the sheaf $K^{p}$.\\

The hypercohomology $\mathbb{H}^i(X; K^{\bullet})$ of a DGS $K^{\bullet}$ is the cohomology of the double complex $\Gamma(X, I^{\bullet\bullet})$, where $\Gamma$ is the global section functor
\begin{equation*}
\mathbb{H}^i(X; K^{\bullet})=H^i\Gamma(X, I^{\bullet\bullet}).
\end{equation*}
\end{definition}
Note that an injective resolution of a bounded-below DGS exists \cite{Bry}.

Although there is a standard way to construct an injective resolution for a sheaf (hence an injective resolution of a DGS), it is in general very hard to explicitly compute the hypercohomology using the injective resolution $I^{\bullet\bullet}$ of a DGS. However one may be able to compute the hypercohomology more directly when the DGS $K^{\bullet}$ is good enough.

\begin{theorem} \label{thme2}
Let $K^{\bullet}$ be a bounded-below complex of sheaves on a space $X$. If each $K^p$ is acyclic (i.e., the sheaf cohomology $H^i(X,K^p)=0$ for $i>0$), then the hypercohomology of $K^{\bullet}$ is isomorphic to the cohomology of the complex
\begin{equation*}
\cdots\rightarrow\Gamma(X, K^p)\rightarrow\Gamma(X, K^{p+1})\rightarrow\cdots.
\end{equation*}
\end{theorem}

\begin{definition}
A sheaf $K$ on a space $X$ is soft if the restriction map $\Gamma(X, K)\rightarrow \Gamma(Y, K)$ is surjective for all closed $Y\subset X$, where $\Gamma(Y, K)$ is defined by the limit of $\Gamma(U, K)$ when $U$ runs through the open sets that contain $Y$.
\end{definition}

\begin{theorem} \label{thme5}
If $X$ is a paracompact Hausdorff space, then soft sheaves are acyclic.
\end{theorem}

Thus, if each sheaf $K^p$ in $K^{\bullet}$ is soft or acyclic, then the hypercohomology can be computed from the complex of global sections. We do not have to worry about the injective resolution and the double complex.

\begin{theorem} \label{thme1}
A quasi-isomorphism $\phi:K^{\bullet}\rightarrow L^{\bullet}$ of bounded-below complexes induces isomorphisms
\begin{equation*}
\mathbb{H}^i(X; K^{\bullet})\cong\mathbb{H}^i(X; L^{\bullet})
\end{equation*}
on hypercohomology groups.
\end{theorem}

\begin{remark}
	This theorem says that when computing the hypercohomology of a DGS $K^{\bullet}$, one may replace it by a more convenient (acyclic or soft) DGS $L^{\bullet}$ that is quasi-isomorphic to $K^{\bullet}$.
\end{remark}

\subsection{Equivariant cohomology of slice groupoids}

We apply the preceding sheaf-theoretic results to generalize equivariant cohomology to slice groupoids.

\begin{theorem}[Morita invariance for action groupoids]
	\label{thm:equivalent-action-groupoids}
	Let
	\[
	\mathcal H=H\ltimes N,
	\qquad
	\mathcal G=G\ltimes M
	\]
	be action Lie groupoids, where \(H\) and \(G\) are compact Lie groups.
	If \(\mathcal H\) and \(\mathcal G\) are equivalent Lie groupoids, then
	there is a canonical isomorphism
	\[
	H_H^*(N;\mathbb R)
	\cong
	H_G^*(M;\mathbb R).
	\]
	More generally, the same conclusion holds if
	\(\mathcal H\) and \(\mathcal G\) are Morita equivalent.
\end{theorem}

\begin{proof}
	Recall that the classifying space of an action groupoid is naturally
	homotopy equivalent to the corresponding Borel construction:
	\[
	B(H\ltimes N)
	\simeq
	E_H\times_H N,
	\]
	and
	\[
	B(G\ltimes M)
	\simeq
	E_G\times_G M.
	\]
	
	A groupoid equivalence induces a weak homotopy equivalence between
	classifying spaces \cite{Moe}. Hence an equivalence
	\[
	\mathcal H\longrightarrow\mathcal G
	\]
	induces a weak homotopy equivalence
	\[
	B\mathcal H
	\simeq
	B\mathcal G.
	\]
	The same conclusion holds for a Morita equivalence, since a span of
	groupoid equivalences
	\[
	\mathcal H
	\longleftarrow
	\mathcal K
	\longrightarrow
	\mathcal G
	\]
	induces weak homotopy equivalences
	\[
	B\mathcal H
	\simeq
	B\mathcal K
	\simeq
	B\mathcal G.
	\]
	
	Therefore
	\[
	H^*(E_H\times_H N;\mathbb R)
	\cong
	H^*(E_G\times_G M;\mathbb R).
	\]
	By the definition of equivariant cohomology,
	\[
	H_H^*(N;\mathbb R)
	=
	H^*(E_H\times_H N;\mathbb R)
	\]
	and
	\[
	H_G^*(M;\mathbb R)
	=
	H^*(E_G\times_G M;\mathbb R).
	\]
	Thus
	\[
	H_H^*(N;\mathbb R)
	\cong
	H_G^*(M;\mathbb R).
	\]
\end{proof}

\subsection{The Cartan sheaf of a slice groupoid}
\label{subsec:Cartan-sheaf-slice-groupoid}

Let
\[
\mathcal G:G_1\rightrightarrows G_0
\]
be a slice groupoid and let
$
X=|\mathcal G|
$
be its orbit space.  Fix a compatible system of local linearized
charts
$
\left\{
(U,\Psi_U)
\right\}_{U\in\mathcal B}
$
as in Definition~\ref{def:compatible-local-linearized-charts}.

For every
$
U\in\mathcal B,
$
write
\[
\Psi_U:
\mathcal G|_U
\longrightarrow
G_U\ltimes S_U,
\]
and associate to the chart the Cartan complex
\[
\mathcal A^\bullet(U)
:=
\operatorname{Car}(S_U,G_U).
\]

We emphasize that no restriction morphism
\[
\mathcal A^\bullet(U)
\longrightarrow
\mathcal A^\bullet(V)
\]
is assumed to exist for every inclusion
$
V\subset U.
$
Instead, restriction morphisms are defined only when a compatible
refinement embedding of the corresponding local models is available.

Suppose that
$
V\subset U
$
are chart domains and that there is a refinement embedding
\[
\lambda_{V,U}:
G_V\ltimes S_V
\longrightarrow
G_U\ltimes S_U.
\]
Write the corresponding equivariant pair as
\[
(\jmath_{V,U},i_{V,U}):
(G_V,S_V)
\longrightarrow
(G_U,S_U).
\]
Thus
\[
\jmath_{V,U}:G_V\longrightarrow G_U
\]
is a Lie group homomorphism and
\[
i_{V,U}:S_V\longrightarrow S_U
\]
is a $\jmath_{V,U}$-equivariant smooth embedding.

The refinement embedding determines a Cartan morphism
\[
\rho_{U,V}:
\operatorname{Car}(S_U,G_U)
\longrightarrow
\operatorname{Car}(S_V,G_V)
\]
defined by
\begin{equation}
	\bigl(\rho_{U,V}\alpha\bigr)(\eta)
	=
	i_{V,U}^*
	\left(
	\alpha(d\jmath_{V,U}(\eta))
	\right),
	\qquad
	\eta\in\mathfrak g_V.
\end{equation}

By Proposition~\ref{prop:Cartan-functoriality}, the equivariant pair
\[
(\jmath_{V,U},i_{V,U}):
(G_V,S_V)
\longrightarrow
(G_U,S_U)
\]
induces a morphism of differential graded algebras
\[
\rho_{U,V}:
\operatorname{Car}(S_U,G_U)
\longrightarrow
\operatorname{Car}(S_V,G_V).
\]
In particular,
\[
D_{C,V}\circ\rho_{U,V}
=
\rho_{U,V}\circ D_{C,U}.
\]

We now organize these local Cartan complexes pointwise.

\begin{definition}
	Let
	$
	y\in X.
	$
	The category
	$
	\mathcal{N}_y
	$
	is defined as follows. Its objects are the local linearized charts
	$
	(U,\Psi_U)
	$
	such that
	$
	y\in U.
	$ A morphism
	\[
	U\longrightarrow V
	\]
	in $\mathcal{N}_y$ is represented by a refinement
	\[
	y\in V\subset U
	\]
	together with a refinement embedding
	\[
	\lambda_{V,U}:
	G_V\ltimes S_V
	\longrightarrow
	G_U\ltimes S_U.
	\]
	
	The direction of the morphism is chosen so that the associated Cartan
	morphism has the same direction:
	\[
	\rho_{U,V}:
	\operatorname{Car}(S_U,G_U)
	\longrightarrow
	\operatorname{Car}(S_V,G_V).
	\]
	
	Two refinement morphisms are identified if they agree after passing
	to a further chart neighborhood of $y$.
	
	Composition is induced by composition of refinement embeddings.
\end{definition}

\begin{proposition}
	\label{prop:refinement-category-filtered}
	For every
	$
	y\in X,
	$
	the category
	$
	\mathcal{N}_y
	$
	is filtered.
\end{proposition}

\begin{proof}
	The category is nonempty because the local linearized chart domains
	form a basis of the topology of $X$.
	
	Let
	$
	U \text{\ and\ } V\in\mathcal{N}_y.
	$
	By compatibility of the local linearized charts, there exists
	$
	W\in\mathcal B
	$
	such that
	\[
	y\in W\subset U\cap V
	\]
	and there are refinement embeddings
	\[
	G_W\ltimes S_W
	\longrightarrow
	G_U\ltimes S_U
	\]
	and
	\[
	G_W\ltimes S_W
	\longrightarrow
	G_V\ltimes S_V.
	\]
	Thus there are morphisms
	\[
	U\longrightarrow W
	\qquad\text{and}\qquad
	V\longrightarrow W
	\]
	in $\mathcal{N}_y$.
	
	Suppose now that two parallel refinement morphisms are given.  By
	the local coherence condition in
	Definition~\ref{def:compatible-local-linearized-charts}, after passing
	to a sufficiently small chart neighborhood
	$
	W\ni y,
	$
	the two induced refinement morphisms agree.  Hence the two parallel
	morphisms are equalized after a further refinement.
	Therefore $\mathcal{N}_y$ is filtered.
\end{proof}

The Cartan complexes and the admissible refinement morphisms define a
functor
\[
\mathcal A_y^\bullet:
\mathcal{N}_y
\longrightarrow
\mathbf{CDGA}_{\mathbb R},
\]
given on objects by
\[
U
\longmapsto
\operatorname{Car}(S_U,G_U)
\]
and on morphisms by
\[
U\longrightarrow V
\quad\longmapsto\quad
\rho_{U,V}.
\]

\begin{definition}[Cartan stalk]
	\label{def:Cartan-stalk-slice-groupoid}
	For
	$
	y\in X,
	$
	define the Cartan stalk of the slice groupoid at $y$ by the filtered
	direct limit
	\begin{equation}
		\mathcal C_{\mathcal G,y}^{\bullet}
		:=
		\varinjlim_{U\in\mathcal{N}_y}
		\operatorname{Car}(S_U,G_U).
	\end{equation}
	
	An element of
	$
	\mathcal C_{\mathcal G,y}^{\bullet}
	$
	will be called a \emph{Cartan germ at $y$}.
\end{definition}

More explicitly, a representative of a Cartan germ at $y$ is a pair
$
(U,\alpha),
$
where
$
y\in U\in\mathcal B
$
and
$
\alpha\in\operatorname{Car}(S_U,G_U).
$
Two representatives
$
(U,\alpha)
\text{\ and\ }
(V,\beta)
$
represent the same germ at $y$ if there exists a common refinement
\[
y\in W\subset U\cap V
\]
with refinement morphisms
\[
\rho_{U,W}:
\operatorname{Car}(S_U,G_U)
\longrightarrow
\operatorname{Car}(S_W,G_W)
\]
and
\[
\rho_{V,W}:
\operatorname{Car}(S_V,G_V)
\longrightarrow
\operatorname{Car}(S_W,G_W)
\]
such that
\[
\rho_{U,W}(\alpha)
=
\rho_{V,W}(\beta).
\]
The equivalence class of $(U,\alpha)$ is denoted by
$
[\alpha]_{y,U},
$
or simply by
$
[\alpha]_y.
$
Because the category $\mathcal{N}_y$ is filtered, this relation is an
equivalence relation.

Since every refinement morphism commutes with the Cartan differential,
the differential passes to the filtered direct limit:
\[
D_C:
\mathcal C_{\mathcal G,y}^{k}
\longrightarrow
\mathcal C_{\mathcal G,y}^{k+1},
\]
with
\[
D_C[\alpha]_y
=
[D_C\alpha]_y.
\]
The product also passes to the direct limit.  Hence
$
\mathcal C_{\mathcal G,y}^{\bullet}
$
is a differential graded algebra.

We now construct the associated espace \'etal\'e. For every degree $k$, set
\[
E_{\mathcal G}^k
=
\bigsqcup_{y\in X}
\mathcal C_{\mathcal G,y}^k
\]
and define
\[
p^k:
E_{\mathcal G}^k
\longrightarrow
X
\]
by
\[
p^k([\alpha]_y)=y.
\]
Let
$
U\in\mathcal B
$
and
$ 
\alpha\in\operatorname{Car}^k(S_U,G_U).
$
Define
\[
\widehat\alpha:
U\longrightarrow
E_{\mathcal G}^k
\]
by
\begin{equation}
	\widehat\alpha(y)
	=
	[\alpha]_y.
\end{equation}

We give
$
E_{\mathcal G}^k
$
the topology generated by the subsets
\[
\widehat\alpha(U)
=
\left\{
[\alpha]_y
\mid
y\in U
\right\},
\]
where
$
U\in\mathcal B
$
and
$
\alpha\in\operatorname{Car}^k(S_U,G_U).
$
The common-refinement property implies that these subsets form a
basis for a topology.  Indeed, if
\[
[\alpha]_y=[\beta]_y,
\]
then, by the definition of equality of germs, there exists a common
refinement
$
y\in W
$
on which the two representatives agree.  Consequently,
\[
\widehat\alpha(W)
=
\widehat\beta(W).
\]

With this topology,
\[
p^k:
E_{\mathcal G}^k
\longrightarrow
X
\]
is a local homeomorphism.

\begin{definition}[Cartan sheaf of a slice groupoid]
	\label{def:slice-groupoid-Cartan-sheaf}
	For every open subset
	$
	O\subset X,
	$
	define
	$
	\mathcal C_{\mathcal G}^k(O)
	$
	to be the space of continuous sections of the espace \'etal\'e
	\[
	p^k:E_{\mathcal G}^k\longrightarrow X
	\]
	over $O$:
	\[
	\mathcal C_{\mathcal G}^k(O)
	=
	\left\{
	s:O\longrightarrow E_{\mathcal G}^k
	\;\middle|\;
	p^k\circ s=\operatorname{id}_O,
	\quad
	s\text{ is continuous}
	\right\}.
	\]
	
	If
	$
	V\subset O
	$
	are arbitrary open subsets of $X$, define the restriction morphism by
	ordinary restriction of sections:
	\[
	\operatorname{res}_{O,V}:
	\mathcal C_{\mathcal G}^k(O)
	\longrightarrow
	\mathcal C_{\mathcal G}^k(V),
	\]
	\[
	\operatorname{res}_{O,V}(s)
	=
	s|_V.
	\]
	
	Thus the restriction morphism is now defined for every inclusion
$
	V\subset O
	$
	and does not require a comparison map between slice charts.
	
	The stalkwise Cartan differential defines
	\[
	D_C:
	\mathcal C_{\mathcal G}^k
	\longrightarrow
	\mathcal C_{\mathcal G}^{k+1}
	\]
	by
	\[
	(D_Cs)(y)
	=
	D_C(s(y)).
	\]
	
	Since every section is locally represented by an element of a local
	Cartan complex, the section $D_Cs$ is again continuous.  Therefore
	$
	(\mathcal C_{\mathcal G}^{\bullet},D_C)
	$
	is a complex of sheaves. The product is defined stalkwise, and hence
	\[
	\mathcal C_{\mathcal G}^{\bullet}
	\]
	is a sheaf of differential graded algebras.
\end{definition}
By construction, the stalk of this sheaf at
\[
y\in X
\]
is naturally identified with the filtered direct limit
\[
(\mathcal C_{\mathcal G}^{\bullet})_y
\cong
\mathcal C_{\mathcal G,y}^{\bullet}
=
\varinjlim_{U\in\mathcal{N}_y}
\operatorname{Car}(S_U,G_U).
\]
For every chart domain
\[
U\in\mathcal B,
\]
there is a natural morphism
\[
\operatorname{Car}(S_U,G_U)
\longrightarrow
\Gamma
\left(
U,\mathcal C_{\mathcal G}^{\bullet}
\right)
\]
given by
\[
\alpha
\longmapsto
\widehat\alpha,
\]
where
\[
\widehat\alpha(y)
=
[\alpha]_y.
\]

\begin{definition}
	\label{def:slice-groupoid-Cartan-model}
	The Cartan complex of the slice groupoid $\mathcal G$ is the complex
	of global sections
	\[
	\operatorname{Car}(\mathcal G)
	:=
	\Gamma
	\left(
	X,\mathcal C_{\mathcal G}^{\bullet}
	\right).
	\]
	The equivariant cohomology of $\mathcal G$ is defined by
	\[
	H_{\mathcal G}^*(\mathcal G;\mathbb R)
	:=
	H^*
	\left(
	\operatorname{Car}(\mathcal G),D_C
	\right).
	\]
\end{definition}

Now let \(G\) be a compact Lie group acting smoothly and properly on
a manifold \(M\), and let
\[
\pi:M\longrightarrow X=M/G
\]
be the quotient map. For every open subset \(U\subset X\), define
\[
\mathcal C_{G,M}^{\bullet}(U)
=
\operatorname{Car}(\pi^{-1}(U),G).
\]
The usual restriction of differential forms defines the restriction
maps. Hence
\[
U\longmapsto
\operatorname{Car}(\pi^{-1}(U),G)
\]
is a sheaf of complexes on \(X\), which we call the ordinary Cartan
sheaf of the \(G\)-manifold \(M\). Its global sections are
\[
\Gamma
\left(
X,\mathcal C_{G,M}^{\bullet}
\right)
=
\operatorname{Car}(M,G).
\]

\begin{proposition}
	\label{prop:Cartan-sheaf-quasi-isomorphism}
	Let \(G\) be a compact Lie group acting smoothly and properly on
	\(M\), and let
	\[
	X=M/G.
	\]
	Let \(\mathcal G\) be a slice groupoid representation of the action
	groupoid \(G\ltimes M\), constructed from a compatible system of local linearized charts in the sense of Definition~\ref{def:compatible-local-linearized-charts}.
	
	Then the local Cartan reduction morphisms define a morphism of
	complexes of sheaves
	\[
	\mathcal R:
	\mathcal C_{G,M}^{\bullet}
	\longrightarrow
	\mathcal C_{\mathcal G}^{\bullet}.
	\]
	Moreover, \(\mathcal R\) is a quasi-isomorphism of complexes of
	sheaves.
\end{proposition}

\begin{proof}
	We first construct the morphism of sheaves.
	Let
	$
	O\subset X
	$
	be open and let
	\[
	\alpha\in
	\mathcal C_{G,M}^{\bullet}(O)
	=
	\operatorname{Car}(\pi^{-1}(O),G).
	\]
	For every point
	$
	y\in O,
	$
	choose a chart domain
	$
	U\in\mathcal B
	$
	such that
	$
	y\in U\subset O.
	$
	For the slice-groupoid representation associated with the
	$G$-action, the chart over $U$ is represented by a slice model
	$
	G_U\ltimes S_U.
	$
	The slice reduction theorem gives a chain map
	\[
	\mathcal R_U:
	\operatorname{Car}(\pi^{-1}(U),G)
	\longrightarrow
	\operatorname{Car}(S_U,G_U).
	\]
	Define
	$
	(\mathcal R_O\alpha)(y)
	$
	to be the germ represented by
	$
	\mathcal R_U
	\left(
	\alpha|_{\pi^{-1}(U)}
	\right).
	$
	Thus
	\begin{equation}
		(\mathcal R_O\alpha)(y)
		=
		\left[
		\mathcal R_U
		\left(
		\alpha|_{\pi^{-1}(U)}
		\right)
		\right]_y.
	\end{equation}
	
	We must show that this definition is independent of the chosen chart
	$U$.
	Suppose
	$
	U,V\in\mathcal B
	$
	both contain $y$ and are contained in $O$.  By the compatibility of
	the local linearized charts, there exists a common refinement
	\[
	y\in W\subset U\cap V
	\]
	together with refinement embeddings from the model over $W$ to the
	models over $U$ and $V$.
	
	Naturality of the Cartan reduction morphism gives
	\[
	\rho_{U,W}
	\circ
	\mathcal R_U
	=
	\mathcal R_W
	\circ
	\operatorname{res}_{U,W}
	\]
	and
	\[
	\rho_{V,W}
	\circ
	\mathcal R_V
	=
	\mathcal R_W
	\circ
	\operatorname{res}_{V,W}.
	\]
	Therefore
	\[
	\rho_{U,W}
	\left(
	\mathcal R_U
	(\alpha|_{\pi^{-1}(U)})
	\right)
	=
	\mathcal R_W
	(\alpha|_{\pi^{-1}(W)})
	\]
	and
	\[
	\rho_{V,W}
	\left(
	\mathcal R_V
	(\alpha|_{\pi^{-1}(V)})
	\right)
	=
	\mathcal R_W
	(\alpha|_{\pi^{-1}(W)}).
	\]
	Hence the two representatives determine the same germ at $y$. Thus
	$
	(\mathcal R_O\alpha)(y)
	$
	is well defined.
	
	Moreover, in a sufficiently small slice neighborhood $U$, the section
	$
	\mathcal R_O\alpha
	$
	is represented by the single Cartan element
	$
	\mathcal R_U
	(\alpha|_{\pi^{-1}(U)}).
	$
	Therefore it is a continuous section of the espace \'etal\'e.
	Consequently,
	\[
	\mathcal R_O:
	\mathcal C_{G,M}^{\bullet}(O)
	\longrightarrow
	\mathcal C_{\mathcal G}^{\bullet}(O)
	\]
	is well defined.
	
	If
	$
	V\subset O
	$
	is open, the construction immediately gives
	\[
	\operatorname{res}_{O,V}
	\circ
	\mathcal R_O
	=
	\mathcal R_V
	\circ
	\operatorname{res}_{O,V}.
	\]
	Therefore the morphisms $\mathcal R_O$ define a morphism of complexes
	of sheaves
	\[
	\mathcal R:
	\mathcal C_{G,M}^{\bullet}
	\longrightarrow
	\mathcal C_{\mathcal G}^{\bullet}.
	\]
	Since each local slice reduction morphism commutes with the Cartan
	differential, one has
	\[
	D_C\circ\mathcal R
	=
	\mathcal R\circ D_C.
	\]
	
	It remains to prove that $\mathcal R$ is a quasi-isomorphism. Fix
	$
	y\in X.
	$
	The stalk of the ordinary Cartan sheaf is
	\[
	(\mathcal C_{G,M}^{\bullet})_y
	=
	\varinjlim_{y\in O}
	\operatorname{Car}(\pi^{-1}(O),G).
	\]
	Because the slice-chart domains form a basis for the topology of $X$,
	the direct limit may be computed on the cofinal system of slice
	neighborhoods:
	\[
	(\mathcal C_{G,M}^{\bullet})_y
	\cong
	\varinjlim_{U\in\mathcal{N}_y}
	\operatorname{Car}(\pi^{-1}(U),G).
	\]
	
	By construction of the Cartan sheaf of the slice groupoid,
	\[
	(\mathcal C_{\mathcal G}^{\bullet})_y
	\cong
	\varinjlim_{U\in\mathcal{N}_y}
	\operatorname{Car}(S_U,G_U).
	\]
	
	The induced morphism on stalks is therefore the filtered direct limit
	of the local slice reduction morphisms
	\[
	\mathcal R_U:
	\operatorname{Car}(\pi^{-1}(U),G)
	\longrightarrow
	\operatorname{Car}(S_U,G_U).
	\]
	For every
	$
	U\in\mathcal{N}_y,
	$
	Theorem~\ref{thm:slice-reduction-quasi-isomorphism} implies that
	$
	\mathcal R_U
	$
	is a quasi-isomorphism.
	
	Since $\mathcal{N}_y$ is filtered and filtered direct limits of real
	vector spaces are exact, cohomology commutes with these direct limits.
	Hence
	\begin{align*}
		H^q
		\left(
		(\mathcal C_{G,M}^{\bullet})_y
		\right)
		&\cong
		\varinjlim_{U\in\mathcal{N}_y}
		H^q
		\left(
		\operatorname{Car}(\pi^{-1}(U),G)
		\right)
		\\
		&\xrightarrow{\cong}
		\varinjlim_{U\in\mathcal{N}_y}
		H^q
		\left(
		\operatorname{Car}(S_U,G_U)
		\right)
		\\
		&\cong
		H^q
		\left(
		(\mathcal C_{\mathcal G}^{\bullet})_y
		\right).
	\end{align*}
	Thus the induced morphism
	\[
	H^q
	\left(
	(\mathcal C_{G,M}^{\bullet})_y
	\right)
	\longrightarrow
	H^q
	\left(
	(\mathcal C_{\mathcal G}^{\bullet})_y
	\right)
	\]
	is an isomorphism for every
	$
	q
	$
	and every
	$
	y\in X.
	$
	
	Therefore
	\[
	\mathcal R:
	\mathcal C_{G,M}^{\bullet}
	\longrightarrow
	\mathcal C_{\mathcal G}^{\bullet}
	\]
	is a quasi-isomorphism of complexes of sheaves.
\end{proof}

\begin{proposition}
	\label{prop:Cartan-sheaves-fine}
	Assume that \(X=M/G\) is paracompact and that for every locally finite
	open covering
	$
	\{U_\alpha\}
	$
	of \(X\), there exists a subordinate partition of unity
	$
	\{\rho_\alpha\}
	$
	such that every pullback
	$
	\pi^*\rho_\alpha
	$
	is a smooth \(G\)-invariant function on \(M\).
	Then, for every degree \(q\), the sheaves
	$
	\mathcal C_{G,M}^{q}\text{\ and\ }
	\mathcal C_{\mathcal G}^{q}
	$
	are fine. In particular, they are acyclic.
\end{proposition}

\begin{proof}
	Let
	$
	\{U_\alpha\}
	$
	be a locally finite open covering of \(X\), and choose a subordinate
	partition of unity
	$
	\{\rho_\alpha\}
	$
	as in the hypothesis.
	
	For the ordinary Cartan sheaf, define
	\[
	h_\alpha:
	\mathcal C_{G,M}^{q}
	\longrightarrow
	\mathcal C_{G,M}^{q}
	\]
	locally by
	\[
	h_\alpha(\omega)
	=
	(\pi^*\rho_\alpha)\omega.
	\]
	Since \(\pi^*\rho_\alpha\) is smooth and \(G\)-invariant,
	multiplication by \(\pi^*\rho_\alpha\) preserves the Cartan complex.
	Moreover,
	\[
	\operatorname{supp}(h_\alpha)
	\subset U_\alpha
	\]
	and
	\[
	\sum_\alpha h_\alpha
	=
	\operatorname{id}.
	\]
	Thus \(\mathcal C_{G,M}^{q}\) is fine. For the slice-groupoid Cartan sheaf, let
$
(U,G_U\ltimes S_U)
$
be a local linearized chart. The function
$
\rho_\alpha|_U
$
pulls back along the quotient map
\[
q_U:S_U\longrightarrow S_U/G_U\simeq U
\]
to a smooth $G_U$-invariant function
$
q_U^*(\rho_\alpha|_U)
$
on $S_U$. Multiplication by this function defines an endomorphism
\[
h_{\alpha,U}:
\operatorname{Car}^q(S_U,G_U)
\longrightarrow
\operatorname{Car}^q(S_U,G_U)
\]
given by
\[
h_{\alpha,U}(\omega)
=
q_U^*(\rho_\alpha|_U)\,\omega.
\]

We show that this operation is compatible with Cartan germs.
Suppose that
$
V\subset U
$
is an admissible refinement with equivariant embedding
\[
i_{V,U}:S_V\longrightarrow S_U.
\]
Since the refinement embedding induces the inclusion
\[
V\hookrightarrow U
\]
on orbit spaces, one has
\[
q_U\circ i_{V,U}
=
\iota_{V,U}\circ q_V.
\]
Consequently,
\[
i_{V,U}^*
q_U^*(\rho_\alpha|_U)
=
q_V^*(\rho_\alpha|_V).
\]

It follows that
\[
\rho_{U,V}
\left(
q_U^*(\rho_\alpha|_U)\,\omega
\right)
=
q_V^*(\rho_\alpha|_V)\,
\rho_{U,V}(\omega).
\]

Therefore multiplication by $\rho_\alpha$ preserves the equivalence
relation defining Cartan germs.  Hence it induces a well-defined
endomorphism on every stalk
$
(\mathcal C_{\mathcal G}^{q})_y.
$
Equivalently, for a germ represented by
\[
[\omega]_y,
\qquad
\omega\in\operatorname{Car}^q(S_U,G_U),
\]
define
\[
\widetilde h_\alpha([\omega]_y)
=
\left[
q_U^*(\rho_\alpha|_U)\,\omega
\right]_y.
\]

This is independent of the choice of representative.  Since the
operation is locally represented by multiplication in a local Cartan
complex, it defines a sheaf endomorphism
\[
\widetilde h_\alpha:
\mathcal C_{\mathcal G}^{q}
\longrightarrow
\mathcal C_{\mathcal G}^{q}.
\]
Moreover,
\[
\operatorname{supp}(\widetilde h_\alpha)
\subset U_\alpha
\]
and
\[
\sum_\alpha\widetilde h_\alpha
=
\operatorname{id}.
\]
Hence
$
\mathcal C_{\mathcal G}^{q}
$
is fine.
	
	Fine sheaves on a paracompact Hausdorff space are soft and therefore
	acyclic. Thus both sheaves are acyclic in every degree.
\end{proof}

Let $G\ltimes M$ be an action groupoid, and $\mathcal{G}$ be a slice groupoid which represents $G\ltimes M$ given as in example \ref{exa c1}, we will show that the ordinary equivariant cohomology of $G\ltimes M$ is isomorphic to the equivariant cohomology of $\mathcal{G}$. 

Note that as in the example \ref{exa c1}, we assume that $G$ acts on $M$ properly. Also, to simplify the computation, we assume that $M/G$ is paracompact, and for any covering of $M/G$ there is a partition of unity of continuous functions subordinate to the covering, and these functions are $C^1$ when being pulled back to $M$ by the projection $M\rightarrow M/G$.

\begin{theorem}
	\label{thm:slice-groupoid-equivariant-cohomology}
	Let \(G\) be a compact Lie group acting smoothly and properly on a
	manifold \(M\), and let
	$
	X=M/G.
	$
	Let \(\mathcal G\) be a slice groupoid representation of the action
	groupoid
	$
	G\ltimes M
	$
	associated with a compatible family of slices.
	
	Assume that \(X\) is paracompact Hausdorff and that every locally
	finite open covering of \(X\) admits a subordinate partition of unity
	whose pullback to \(M\) consists of smooth \(G\)-invariant functions.
	
	Then there is a natural isomorphism
	\[
	H_{\mathcal G}^*(\mathcal G;\mathbb R)
	\cong
	H_G^*(M;\mathbb R).
	\]
\end{theorem}

\begin{proof}
	Let
	$
	\mathcal C_{G,M}^{\bullet}
	$
	be the ordinary Cartan sheaf on
	$
	X=M/G,
	$
	and let
	$
	\mathcal C_{\mathcal G}^{\bullet}
	$
	be the Cartan sheaf of the slice groupoid. By Proposition~\ref{prop:Cartan-sheaf-quasi-isomorphism}, the local
	slice reduction morphisms define a quasi-isomorphism of complexes of
	sheaves
	\[
	\mathcal R:
	\mathcal C_{G,M}^{\bullet}
	\longrightarrow
	\mathcal C_{\mathcal G}^{\bullet}.
	\]
	Therefore the induced map on hypercohomology is an isomorphism:
	\[
	\mathbb H^*
	\left(
	X;
	\mathcal C_{G,M}^{\bullet}
	\right)
	\cong
	\mathbb H^*
	\left(
	X;
	\mathcal C_{\mathcal G}^{\bullet}
	\right).
	\]
	
	By Proposition~\ref{prop:Cartan-sheaves-fine}, each sheaf
	$
	\mathcal C_{G,M}^{q}
	\text{\ and\ }
	\mathcal C_{\mathcal G}^{q}
	$
	is fine, hence acyclic. Therefore the hypercohomology of each complex
	is computed by its complex of global sections:
	\[
	\mathbb H^*
	\left(
	X;
	\mathcal C_{G,M}^{\bullet}
	\right)
	\cong
	H^*
	\left(
	\Gamma
	\left(
	X,\mathcal C_{G,M}^{\bullet}
	\right)
	\right)
	\]
	and
	\[
	\mathbb H^*
	\left(
	X;
	\mathcal C_{\mathcal G}^{\bullet}
	\right)
	\cong
	H^*
	\left(
	\Gamma
	\left(
	X,\mathcal C_{\mathcal G}^{\bullet}
	\right)
	\right).
	\]
	
	By construction,
	\[
	\Gamma
	\left(
	X,\mathcal C_{G,M}^{\bullet}
	\right)
	=
	\operatorname{Car}(M,G),
	\]
	whereas Definition~\ref{def:slice-groupoid-Cartan-model} gives
	\[
	\Gamma
	\left(
	X,\mathcal C_{\mathcal G}^{\bullet}
	\right)
	=
	\operatorname{Car}(\mathcal G).
	\]
	Consequently,
	\[
	H^*
	\left(
	\operatorname{Car}(M,G)
	\right)
	\cong
	H^*
	\left(
	\operatorname{Car}(\mathcal G)
	\right).
	\]
	
	Since \(G\) is compact, Theorem~\ref{thma4} gives
	\[
	H^*
	\left(
	\operatorname{Car}(M,G)
	\right)
	\cong
	H_G^*(M;\mathbb R).
	\]
	By the definition of the equivariant cohomology of the slice
	groupoid,
	\[
	H^*
	\left(
	\operatorname{Car}(\mathcal G)
	\right)
	=
	H_{\mathcal G}^*(\mathcal G;\mathbb R).
	\]
	Hence
	\[
	H_G^*(M;\mathbb R)
	\cong
	H_{\mathcal G}^*(\mathcal G;\mathbb R).
	\]
\end{proof}

This theorem shows that the equivariant cohomology of slice groupoids indeed generalizes the ordinary equivariant cohomology of Lie group actions on manifolds. 


\end{document}